\documentclass[journal,11pt,doublespace,onecolumn]{IEEEtran}

\usepackage{graphicx,psfrag,epsf}
\usepackage{enumerate}
\usepackage{url} 
\usepackage{amssymb,amsfonts,mathtools,amsmath}
\usepackage{bm}
\usepackage{color, colortbl}
\usepackage{float}
\usepackage{algorithm}
\usepackage{algorithmic}
\usepackage{cite,balance}
\usepackage[lofdepth,lotdepth]{subfig}
\usepackage{booktabs}

\makeatletter
\def\th@plain{%
  \thm@notefont{}
  \itshape 
}
\def\th@definition{%
  \thm@notefont{}
  \normalfont 
}
\makeatother

\DeclareMathOperator*{\argmin}{arg\,min}
\newcommand{\rom}[1]{\uppercase\expandafter{\romannumeral #1\relax}}
\DeclarePairedDelimiterX{\norm}[1]{\lVert}{\rVert}{#1}
\DeclarePairedDelimiterX{\bnorm}[1]{\biggl\lVert}{\biggr\rVert}{#1}
\DeclarePairedDelimiterX{\abs}[1]{\lvert}{\rvert}{#1}

\newtheorem{definition}{Definition}
\newtheorem{remark}{Remark}
\newtheorem{theorem}{Theorem}

\newtheorem{lemma}{Lemma} 
\newtheorem{corollary}{Corollary}

\newtheorem{example}{Example}
\newtheorem{assumption}{Assumption}
\newtheorem{proposition}{Proposition}

\def\N{\mathbb{N}}
\def\Real{\mathbb{R}}
\def\G{\mathbb{G}_n}

\def\T{{ \mathrm{\scriptscriptstyle T} }}
\def\P{{ \mathrm{pr} }}
\def\v{{\varepsilon}}
\def\l{l} 
\def\L{\mathcal{L}_n} 
\def\R{\mathcal{R}_n} 
\def\Real{\mathbb{R}}
\def\Lhat{\hat{\mathcal{L}}_n}
\def\E{\mathbb{E}_{*}}

\def\H{\mathcal{H}_n}

\def\Z{\mathcal{Z}}
\def\F{\mathcal{F}}
\def\P{P_*}
\def\de{\overset{\Delta}{=}}

\def\limp{\rightarrow_{p}}
\def\limd{\rightarrow_{d}}
\def\th{\bm \theta} 
\def\thE{\hat{\bm \theta}_n} 
\def\thT{\bm \theta^*_n} 
\def\thTT{\bm \theta^*} 
\def\D{\mathcal{Z}} 
\def\mo{\alpha} 
\def\Mo{\mathcal{A}_n} 
\def\MoT{\mathcal{A}_T} 
\def\p{d_n} 
\def\d{d_n[\mo]} 
\def\tr{\textit{tr}}
\def\card{\textrm{card}}
\def\var{\textrm{Var}}
\def\size{\textrm{size}}
\def\S{\Sigma}
\def\Lc{\mathcal{L}_t^c}
\def\Lcn{\mathcal{L}_n^c}
\def\INT{\textrm{int}}
\def\yt{\bm z_t}
\def\ys{\bm z_s}
\def\deg{D} 

\makeatletter
\newcommand{\doublewidetilde}[1]{{%
  \mathpalette\double@widetilde{#1}%
}}
\newcommand{\double@widetilde}[2]{%
  \sbox\z@{$\m@th#1\widetilde{#2}$}%
  \ht\z@=.9\ht\z@
  \widetilde{\box\z@}%
}
\makeatother

\newcommand{\ls}{l}
\newcommand{\yy}{\mathcal{Y}}
\newcommand{\RR}{\mathbb{R}}
\newcommand{\bt}{\beta}
\newcommand{\ind}[1]{1_{#1}}

\def\co{}

\begin{document}

\title{On Statistical Efficiency in Learning}

\author{Jie~Ding,~
		Enmao~Diao, 
		Jiawei~Zhou,~and
        Vahid~Tarokh
\thanks{J.~Ding is with the School of Statistics, University of Minnesota, Minneapolis, Minnesota 55455, USA.
E.~Diao and V.~Tarokh are with the Department of Electrical and Computer Engineering, Duke University, Durham, North Carolina 27708, USA.
J.~Zhou is with the John A. Paulson School of Engineering and Applied Sciences, Harvard University, Cambridge, MA, 02138 USA.}
\thanks{This research was funded by the Army Research Office (ARO) under grant number W911NF-20-1-0222.}
\thanks{This paper was presented in part at the 2018 IEEE International Conference on Acoustics, Speech and Signal Processing.}

}

\markboth{IEEE Transactions on Information Theory 
}
{Shell \MakeLowercase{\textit{et al.}}: IEEE Transactions on Information Theory }

\maketitle

\begin{abstract}
A central issue of many statistical learning problems is to select an appropriate model from a set of candidate models. Large models tend to inflate the variance (or overfitting), while small models tend to cause biases (or underfitting) for a given fixed dataset. In this work, we address the critical challenge of model selection to strike a balance between model fitting and model complexity, thus gaining reliable predictive power.  We consider the task of approaching the theoretical limit of statistical learning, meaning that the selected model has the predictive performance that is as good as the best possible model given a class of potentially misspecified candidate models. We propose a generalized notion of Takeuchi's information criterion and prove that the proposed method can asymptotically achieve the optimal out-sample prediction loss under reasonable assumptions. It is the first proof of the asymptotic property of Takeuchi's information criterion to our best knowledge. Our proof applies to a wide variety of nonlinear models, loss functions, and high dimensionality (in the sense that the models' complexity can grow with sample size). The proposed method can be used as a computationally efficient surrogate for leave-one-out cross-validation. Moreover, for modeling streaming data, we propose an online algorithm that sequentially expands the model complexity to enhance selection stability and reduce computation cost. Experimental studies show that the proposed method has desirable predictive power and significantly less computational cost than some popular methods. 
\end{abstract}


\section{Introduction}

How much knowledge can we learn from a given set of data?
Statistical modeling provides a simplification of real-world complexity. It can be used to learn key representations from available data and to predict future data.
To model the data, typically the first step in data analysts is to narrow the scope by specifying a set of candidate parametric models (referred to as model class). 
The model class can be determined by exploratory studies or scientific reasoning.
For data with specific types and sizes, each postulated model may have its advantages. In the second step, data analysts estimate the parameters and fitting performance of each candidate model. 
An illustration of a typical learning procedure is plotted in Fig.~\ref{fig:learning}, where the underlying data generating process may or may not be included in the model class.
Selecting the model with the best fitting performance usually leads to suboptimal results. For example, the largest model always fits the best in a nested model class. But an overly complex model can lead to inflated variance in parameter estimation and thus overfitting. 
Therefore, the third step is to apply a suitable model selection procedure, which will be elaborated in the next section. 

\begin{figure*}[h]
\centering
 \includegraphics[width=0.7\linewidth]{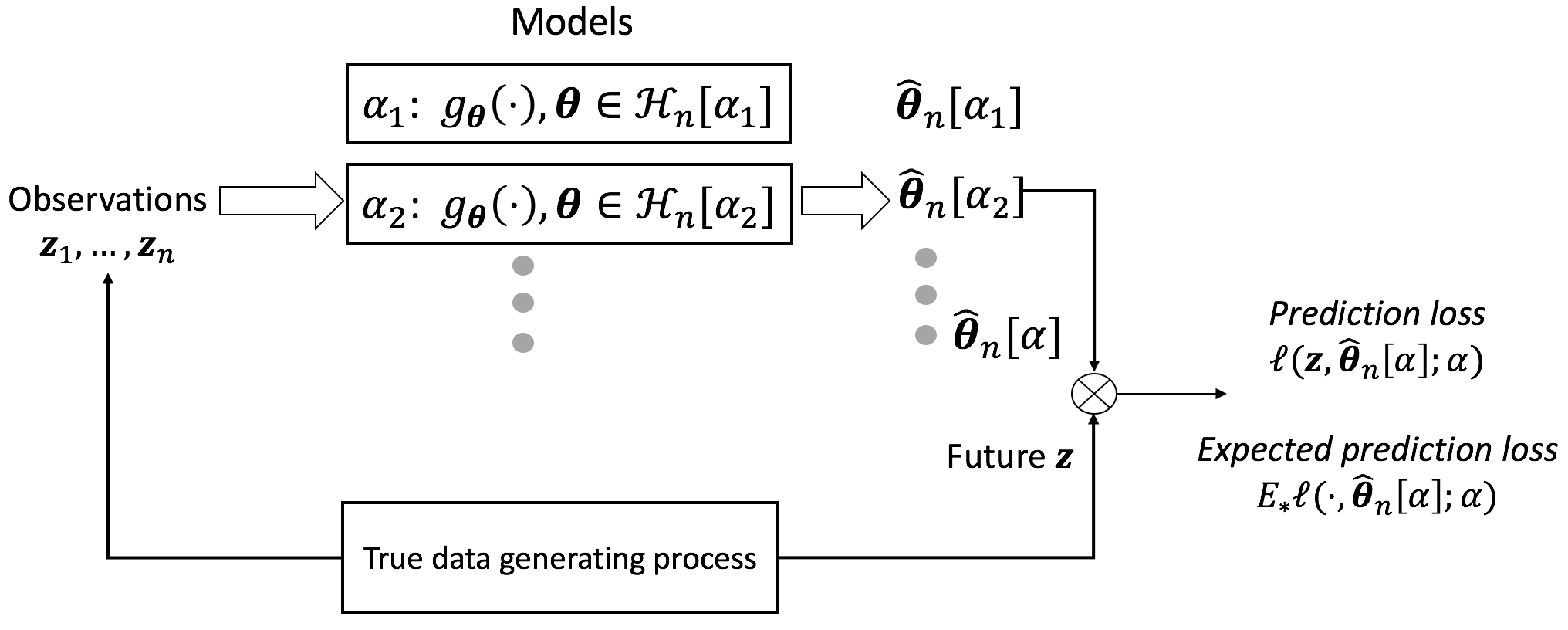}
 \vspace{-0.0in}
 \caption{Illustration of a typical learning procedure, where each candidate model $\mo_j$ is trained in terms of $\thE[\mo_j]$ in its parameter space $\mathcal{H}_n[\mo_j]$, and then used to evaluate future data under some loss function $\ell(\cdot)$. }
 \label{fig:learning}
 \vspace{-0.0in}
\end{figure*}


How can we quantify the theoretical limits of learning procedures? We first introduce the following definition that quantifies the predictive power of each candidate model. 

\begin{definition}[Out-sample prediction loss] \label{def:loss}
The loss function for each sample size $n$ and $\mo \in \Mo$ (model class) is a map $\l(\cdot,\cdot ; \mo): \D \times \H[\mo] \rightarrow \Real$, usually written as $\l(\bm z, \th; \mo)$, where $\D$ is the data domain, $\H[\mo]$ is the parameter space associated with model $\mo$, and $\mo$ is included to emphasize the model under consideration. 
As Fig.~\ref{fig:learning} shows, for a loss function and a given dataset $\bm z_1,\ldots,\bm z_n$ which are independent and identically distributed (i.i.d.), each candidate model $\mo$ produces an estimator $\thE[\mo]$ (referred to as the minimum loss estimator) defined by 
\begin{align}
	\thE[\mo] \de \argmin_{\theta \in \H[\mo]}  \frac{1}{n}\sum_{i=1}^n \l(\bm z_i, \th; \mo). \label{eq:MLE}
\end{align}

Moreover, given by candidate model $\mo$, denoted by $\L(\mo)$, the out-sample prediction loss, also referred to as the generalization error in machine learning, is defined by 
	\begin{align}
		\L(\mo) 
			&\de \E \l \bigl(\cdot, \thE[\mo]; \mo \bigr)	\nonumber\\
			&= \int_{\D} p(\bm z) \l \bigl(\bm z, \thE[\mo]; \mo\bigr) d \bm z .\label{eq:loss}
	\end{align}
Here, $\E$ denotes the expectation with respect to the distribution of a future unseen random variable $\bm z$ (conditional on the observed data). 
We also define the risk by 
$$\R[\mo] = \mathbb{E}_{*,o} \L[\mo],$$  
where the expectation in $\R[\mo]$ is taken with respect to the observed data.
\end{definition}

{\co 
The above notation applies to both supervised and unsupervised learning. In supervised learning, $\bm z$ often consists of a label $\bm y$ and feature $\bm x$, and only the entries of $\bm x$ associated with $\mo$ are involved in the evaluation of $\l$. 
In Statistical Learning Theory, the $\L(\mo) $ may be written as $R(f)= \E \ell(f(X), Y)$, where $f$ is the estimated function under $\mo$. We note that the loss function $\ell$ here is not tied to a particular parametrization of $f$. In comparison, our earlier notion of $\ell$ involves model parameters to develop technical results in this paper. The parameterization may represent regression coefficients from basis expansions (e.g., polynomials, splines, and wavelets) or distributional parameters in a finite mixture model.
}

Throughout the paper, we consider loss functions $\l(\cdot)$ such that $\L[\mo]$ is always nonnegative. A common choice is to use negative log-likelihood of model $\mo$ minus that of the true data-generating model (or its closest candidate model). Table~\ref{table:loss} lists some other loss functions widely used in machine learning. We also provide two motivating examples below.

\begin{example}[Generalized linear models] \label{example:GLM}
	In a generalized linear model (GLM), each response variable $y$ is assumed to be generated from a particular distribution (e.g. Gaussian, Binomial, Poisson, Gamma), with its mean $\bm \mu$ linked with potential covariates $x_1,x_2,\ldots$ through $\E (y)= \mu = g(\beta_1 x_1 + \beta_2 x_2 +\cdots)$ where $g(\cdot)$ is a link function. 
	In this example, data $\bm z=[y,x_1,x_2,\ldots]^\T$, unknown parameters are $\bm \theta = [\beta_1,\beta_2,\ldots]^\T$, and models are subsets of $\{\beta_1,\beta_2,\ldots\}$. 
	We may be interested in the most appropriate distribution family as well as the most significant variables $x_j$'s (\textit{relationships}). 
\end{example}

\begin{example}[Neural networks] \label{example:NN}
	In establishing a neural network (NN) model, we need to choose the number of neurons and hidden layers, activation function, and connectivity configuration.  
	In this example, data are similar to that of the above example, and unknown parameters are the weights on connected edges. 
	Clearly, with a larger number of neurons and connections, more complex functional relationships can be modeled. However, selecting large models may result in overfitting and more computational complexity. 
\end{example}

Based on Definition~\ref{def:loss}, a natural way to define the limit of statistical learning is by using the optimal prediction loss.  
 
\begin{table*}[h]
\centering
\caption{Some common loss functions in addition to negative log-likelihood}
\label{table:loss}
\begin{tabular}{cccccc}
\toprule
Name & quadratic            & exponential           & hinge                   & perceptron               & logistic            
\\ \midrule
Formula    & $(y - \th^\T \bm x)^2$ & $e^{-y \th^\T \bm x}$ & $\max\{0, 1 - y \th^\T \bm x\}$ & $\max\{0, - y \th^\T \bm x\}$ & $\log(1+e^{-y \th^\T \bm x})$ \\ 
\midrule
Domain    & $y \in \Real$ & $y \in \Real$ & $y \in \Real$ & $y \in \Real$ & $y \in \{0,1\}$ 
\\ 
\bottomrule
\end{tabular}
\end{table*}


{\co
\begin{definition}[Oracle performance] \label{def:lol}
	For a given data (of size $n$) and model class $\Mo$, the oracle performance is defined as $\min_{\mo \in \Mo} \L(\mo)$, the optimal out-sample prediction loss offered by candidate models.
\end{definition}
}

The oracle performance is associated with three key elements: data, loss function, and model class. 
Motivated by the original derivation of Akaike information criterion (AIC)~\cite{akaike1970statistical,akaike1998information} and Takeuchi's information criterion (TIC)~\cite{TIC}, we propose a penalized selection procedure and prove adaptivity to the oracle under some regularity assumptions. Those assumptions allow a wide variety of loss functions, model classes (i.e., nested, non-overlapping or partially-overlapping), and high dimensions (i.e., the models' complexity can grow with sample size).
It is worth noting that asymptotic analysis for a fixed number of candidate models with fixed dimensions is generally straightforward. Under some classical regularity conditions (e.g., \cite[Theorem 19.28]{van2000asymptotic}), the likelihood-based principle usually selects the model that attains the smallest Kullback-Leibler divergence from the data generating model. However, our high dimensional setting considers models whose dimensions and parameter spaces may depend on sample size. Thus we cannot directly use the technical tools that have been used in the classical asymptotic analysis for misspecified modes. We will develop some new technical tools in our proof.
Our theoretical results extend the classical statistical theory on AIC for linear (fixed-design) regression models to a broader range of generalized linear or nonlinear models. 
Moreover, we also review the conceptual and technical connections between cross-validation and information criteria.
In particular, we show that the proposed procedure can be much more computationally efficient than cross-validations (with comparable predictive power).

Why is it necessary to consider a high dimensional model class, in the sense that the number of candidate models or each model's complexity is allowed to grow with sample size?
In the context of regression analysis, technical discussions that address the question have been elaborated in \cite{shao1997asymptotic,DingBridge2}. 
Here, we give an intuitive explanation for a general setting. 
We let $\thT[\mo]$ denote the minimum loss parameter defined by  
\begin{align}
	\thT[\mo] &\de \argmin_{\theta \in \H[\mo]} \E \l(\cdot, \th; \mo) .	\label{eq:tMLE}
\end{align}
{\co We note that $\thT[\mo]$ for some models such as neural networks may not be unique. }
Using Taylor expansion under some regularity conditions, $\L[\mo]$ may be expressed as 
\begin{align}
	\L[\mo] =
	&\E \l(\bm z , \thT[\mo] ; \mo) + 
   	\frac{1}{2}\bigl(\thE[\mo]-\thT[\mo]\bigr)^\T V_n(\thT[\mo]; \mo) \bigl(\thE[\mo]-\thT[\mo]\bigr)\times \{1+o_p(1)\} \label{eq80_new}
\end{align}
where 
$V_n(\th; \mo) \de \E \nabla^2_{\th} l_{n}(\cdot, \th; \mo) $, and $o_p(1)$ is a sequence of random variables that converges to zero in probability. 
The main idea of (\ref{eq80_new}) is to expand $\L[\mo]$ at a projection point $\thT[\mo]$ under some uniform convergence condition in its vicinity.  
Theoretical justifications of (\ref{eq80_new}) or its variants for a model whose dimension depends on $n$ have been studied in several earlier work, e.g. in \cite{portnoy1984asymptotic,portnoy1986central,portnoy1988asymptotic}.
The out-sample prediction loss consists of two additive terms: the bias and the variance. 
Large models tend to reduce the bias but inflate the variance (\textit{overfitting}), while small models tend to reduce the variance but increase the bias (\textit{underfitting}) for a given fixed dataset. Suppose that ``all models are wrong'', meaning that the data generating model is not included in the model class. 
Usually, the bias is non-vanishing (with $n$) for a fixed model complexity (say $d$), and it is approximately a decreasing function of $d$; while on the other hand, the variance vanishes at rate $n^{-1}$ for a fixed $d$, and it is an increasing function of $d$. Suppose for example that the bias and variance terms are approximately $c_1 \gamma^{-d}$ and $c_2 d / n$, respectively, for some positive constants $c_1,c_2,\gamma$. Then the optimal $d$ is at the order of $\log(n)$.

In view of the above arguments, as more data become available, the model complexity needs to be enlarged to strike a balance between bias and variance (or \textit{approach the oracle}). 
To illustrate, we generated $n=100,200$ data from a logistic regression model, where coefficients are $\beta_i=10/i$ and covariates $x_i$'s are independent standard Gaussian (for $i=1,\ldots,100$). We consider the nested model class $\Mo=\{\{1\},\{1,2\},\ldots,\{1,2,\ldots,50\}\}$, and the loss function is chosen to be the negative log-likelihood. 
We summarize the results in Fig.~\ref{fig_intro}. 
As model complexity increases, the model fitting as measured by in-sample loss improves (Fig.~\ref{fig_intro1}). In contrast, the predictive power, as measured by the out-sample prediction loss, first improves and then deteriorates after some ``optimal dimension'' (Fig.~\ref{fig_intro2}).
Also, the optimal dimension becomes larger as the sample size increases.  


\begin{figure}[h!]
\centering
\subfloat[Subfigure 1 list of figures text][The fitting performance of each model under sample size $n=100$ (solid blue) and $n=200$ (dash red). ]{
\includegraphics[width=0.3\textwidth]{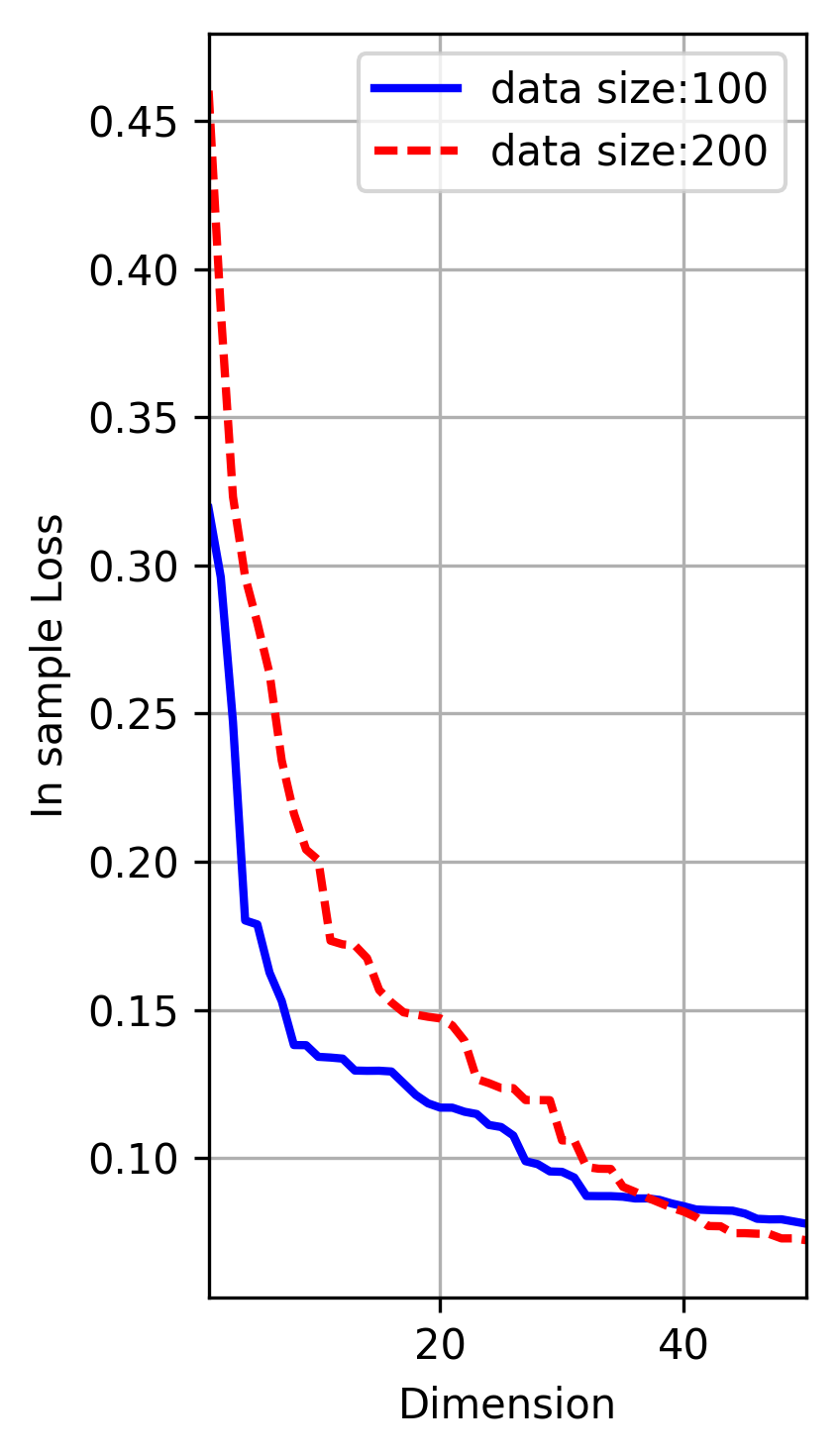}
\label{fig_intro1}}
\quad
\subfloat[Subfigure 2 list of figures text][The out-sample prediction loss (numerically computed using independently generated data) of each model under sample size $n=100$ (solid blue) and $n=200$ (dash red).]{
\includegraphics[width=0.3\textwidth]{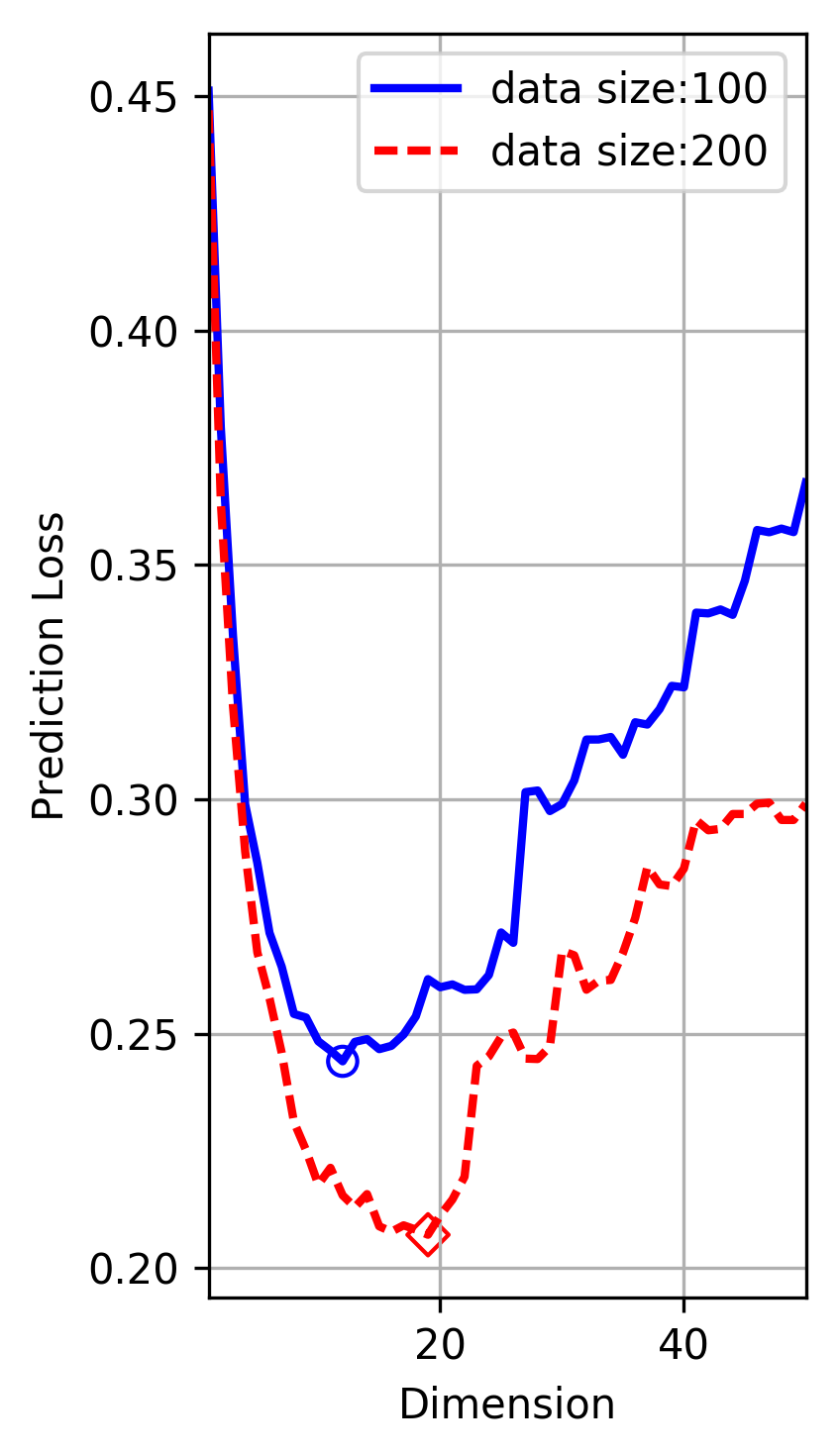}
\label{fig_intro2}}
\caption{Experiment showing the ``bigger models for bigger data'' phenomena that is almost ubiquitous in statistical prediction and machine learning tasks. }
\label{fig_intro}
\end{figure}

As data sequentially arrive, the selected model from our proposed method (and many other existing methods such as cross-validation) suffers from fluctuations due to randomness. 
A conceptually appealing and computationally efficient way is to move from small models to larger models sequentially.
For that purpose, based on the proposed method, we further propose a sequential model expansion strategy that aims to facilitate interpretability of learning.

The outline of the paper is given as follows. 
In Section~\ref{sec:LOL}, we propose a computationally efficient method that determines the most appropriate learning model as more data become available. 
We prove that the oracle performance can be asymptotically approached under some regularity assumptions. 
In Section~\ref{sec:expand}, we propose a model expansion technique building upon a new online learning algorithm, which we refer to as ``graph-based'' learning. The online learning algorithm may be interesting on its own as it exploits graphical structure when updating the expert systems and computing the regrets. 
In a supplementary material available at~\cite{DingLoLExperiment}, we experimentally demonstrate the applications of the proposed methodology to generalized linear models and neural networks in selecting the variables/neurons. 
The related open-sources codes are also provided.

\section{Related Work}

A wide variety of model selection techniques have been proposed in the past fifty years, motivated by different viewpoints and justified under various circumstances. We refer to \cite{barron1999risk,massart2007concentration,claeskens2008model,arlot2010survey,DingOverview} for more surveys. 
This section briefly reviews some closely related work in information criterion and cross-validation, and includes a derivation of TIC.

\subsection{Information Criteria}

Examples of penalized selection include final prediction error criterion~\cite{akaike1969fitting}, AIC~\cite{akaike1970statistical,akaike1998information}, TIC~\cite{TIC}, BIC~\cite{schwarz1978estimating} and its Bayesian counterpart Bayes factor~\cite{casella2009consistency}, minimum description length criterion~\cite{hansen2001model}, Hannan and Quinn criterion~\cite{hannan1979determination}, predictive minimum description length criterion~\cite{rissanen1986stochastic,wei1992predictive}, $C_p$ method~\cite{mallows1973some}, generalized information criterion (GIC$_{\lambda_n}$) with $\lambda_n \rightarrow \infty$~\cite{nishii1984asymptotic,rao1989strongly,shao1997asymptotic}, generalized cross-validation method (GCV)~\cite{craven1978smoothing}, the Goldenshluger-Lepski method~\cite{lepskii1991problem,goldenshluger2008universal,goldenshluger2013general}, and the bridge criterion (BC)~\cite{DingBridge}. 
Recently, a regularization approach named as information criterion estimation (ICE)~\cite{dixon2018takeuchi} is proposed 
that extends TIC to handle non-MLE estimates in over-parameterized models. 
An extension of AIC and Mallows' $C_p$ method 
is the `slope heuristics' approach proposed in~\cite{birge2001gaussian,birge2007minimal} for Gaussian model selection and later developed to more general settings~\cite{baudry2012slope,arlot2019minimal,arlot2019rejoinder}.
The main idea of slope heuristics is to recognize the existence of a minimal penalty such that the 
 out-sample prediction loss of the selected model with lighter penalties explode, and to show that a penalty equal to twice the minimal penalty often enables model selection that meets the 
inequality:
$ 
\L[\hat{\mo}_n] \leq c_n \min_{\mo \in \Mo} \L[\mo] + \eta_n, 
$ 
also called oracle 
inequality, for $c_n$ close to 1 and $\eta_n $ negligible with respect to the value of $\min_{\mo \in \Mo} \L[\mo]$. 
In theory, the asymptotic efficiency is a limiting requirement of the oracle inequality 
with $c_n =1+o_p(1)$ and $\eta_n = o_p(1) \times \min_{\mo \in \Mo} \L[\mo]$ as $n\rightarrow \infty$. 
There have been fruitful results in non-asymptotic quantifications of $c_n$ and $\eta_n$ using concentration inequalities (see e.g.~\cite{boucheron2003concentration,boucheron2003concentration2,boucheron2013concentration,koltchinskii2011introduction,arlot2019minimal,arlot2019rejoinder}). Non-asymptotic analysis is often based on concentration inequalities or Stein's method \cite{massart2007concentration,gaunt2017chi}. In this work, we are not looking for oracle inequalities with non-asymptotic analysis.
On the other hand, the recent development of slope heuristics has motivated the data-driven construction of penalty terms instead of using pre-determined penalty functions. An example in this direction is the dimension jump method~\cite{birge2007minimal,arlot2009data}, which, for a given penalty shape, identifies the suitable multiplicative constant by searching for a significant jump of the selected dimension against different constants.

 
\subsection{Cross-validation (CV)}
The basic idea of cross-validation~\cite{stone1974cross,allen1974relationship} is to split the data into two parts: one for training and one for testing. The model with the best testing performance is selected, hoping that it will perform well for future data. 
It is a common practice to apply a 10-fold CV, 5-fold CV, 2-fold CV, or 30\%-for-testing.
In general, the advantages of the CV method are its stability and easy implementation. 
However, \textit{is cross-validation really the best choice?}

In fact, it has been shown that only the delete-$d$ CV method with $\lim_{n \rightarrow \infty} d/n = 1$~\cite{geisser1975predictive,burman1989comparative,shao1993linear,zhang1993model}, or the delete-$1$ CV method (or leave-one-out, LOO)~\cite{stone1977asymptotic} can exhibit asymptotic optimality.  
Specifically, the former CV exhibits the same asymptotic behavior as BIC, which is typically consistent in a well-specified model class (i.e., it contains the true data generating model), but is suboptimal in a misspecified model class. 
The latter CV is shown to be asymptotically equivalent to AIC/TIC and GCV if $\d = o(n)$~\cite{stone1977asymptotic,shao1997asymptotic}, which is asymptotically efficient in a misspecified model class but usually overfits in a well-specified model class.  
An appropriate choice of the splitting ratio often depends on specific learning tasks, such as the prediction of unobserved data, selection of model, selection of other criteria~\cite{zhang2015cross}, goodness-of-fit test~\cite{DingBAGofT}.  
We refer to \cite{shao1997asymptotic,yang2005can,arlot2010survey,DingBridge,DingOverview} for more detailed discussions on 
the discrepancy and reconciliation of different CVs. 

In particular, for the prediction purpose, common folklore that advocates the use of $k$-fold or 30\%-for-testing CV are asymptotically suboptimal (in the sense of Definition~\ref{def:efficiency}), even in linear regression models~\cite{shao1997asymptotic}. 
Since the only optimal CV is LOO-type in misspecified settings, 
it is more appealing to apply AIC or TIC 
	that gives the same asymptotic performance, and \textit{significantly reduces the computational complexity} by $n$ times. 
For general misspecified nonlinear model class, we shall prove 
	that the GTIC procedure asymptotically approaches the oracle. While the asymptotic performance of LOO is not clear in that case, it is typically more complex to implement. 
To demonstrate that, we shall provide some experimental studies in the supplementary material. As a result, the GTIC procedure can be a promising competitor of various standard CVs adopted in practice.

%

\subsection{Background of TIC}

TIC~\cite{TIC} was heuristically derived as an alternative of AIC, also from an information-theoretic view rooted in Kullback-Leibler (KL) divergence. 
{\co Recall that AIC selects a model that minimizes the negative maximum log-likelihood value plus the model dimension.}
In the seminal work of \cite{stone1977asymptotic}, TIC is shown to be asymptotically equivalent to cross-validation when the purpose is to minimize the KL divergence, and AIC is a special case of TIC when the models under consideration are well-specified.
It appears neither widely appreciated nor used~\cite{burnham2003model} compared with other information criteria such as AIC or Bayesian information criterion (BIC)~\cite{schwarz1978estimating}.
In terms of provable asymptotic performance, only AIC is known to be asymptotically efficient for variable selection in regression models~\cite{shibata1981optimal} and autoregressive order selection in time series models~\cite{shibata1980asymptotically,ing2005orderselection} when models are misspecified. 
Conceptually, TIC was proposed as a surrogate for AIC in general misspecified settings, but the optimality of AIC and TIC in the general context remains unknown.  
As the original paper of TIC~\cite{TIC} was not written in English, we review it for the completeness of the paper. Similar derivations can be found in, e.g., \cite{dixon2018takeuchi}.

Suppose that our goal is to select the model $\mo$ that minimizes logarithmic loss $\E\{-\log p_{\thE[\mo]}(Y)\}$ (or equivalently, minimizes the KL divergence from the true data-generating distribution), where $\thE[\mo]$ is the MLE under model $\mo$. 
For notational convenience, we drop the model index $\mo$ and focus on one model. 
The motivation of TIC was to approximate $\E\{-\log p_{\thE}(\bm z)\}$ by $n^{-1}\sum_{i=1}^n \{-\log p_{\thE}(\bm z_i) \} + \lambda_n$, where the first term is computable from data and the second term is to be asymptotically approximated. 
Under some regularity conditions, the classical sandwich formula of MLE~\cite[Theorem 3.2]{white1982maximum} gives $\sqrt{n}(\thE - \thTT) \limd \mathcal{N}(0, V^{-1}JV^{-1})$ for some $\thTT$ in the parameter space, with 
\begin{align*}
	V=-\E\biggl\{\frac{\partial^2}{\partial\th^2} \log p_{\thTT}(Y)\biggr\}, \,
	J=\E\biggl\{\biggl(\frac{\partial}{\partial\th} \log p_{\thTT}(Y) \biggr) \biggl( \frac{\partial}{\partial\th} \log p_{\thTT}(Y)\biggr)^\T \biggr\}.
\end{align*}
Applying Taylor expansion at $\th = \thTT$, we have 
\begin{align}
	\E\{-\log p_{\thE}(\bm z)\} &\approx \E\{-\log p_{\thTT}(\bm z)\} + \frac{1}{2} (\thE-\thTT)^\T V (\thE-\thTT) \nonumber	\\
	n^{-1}\sum_{i=1}^n \{-\log p_{\thE}(\bm z_i) \} &\approx n^{-1}\sum_{i=1}^n \{-\log p_{\thTT}(\bm z_i) 
	\} - (\thE-\thTT)^\T n^{-1}\sum_{i=1}^n \frac{\partial \log p_{\thTT}(\bm z_i) }{\partial \th} 
	+ \frac{1}{2} (\thE-\thTT)^\T V 
	(\thE-\thTT) 
	\nonumber
\end{align}
and thus
\begin{align}
	\lambda_n = \E\{-\log p_{\thE}(\bm z)\}- n^{-1}\sum_{i=1}^n \{-\log p_{\thE}(\bm z_i) \}
	\approx (\thE-\thTT)^\T n^{-1}\sum_{i=1}^n \frac{\partial \log p_{\thTT}(\bm z_i) }{\partial \th} 
\end{align}
for large $n$. 
Using  
 $$
n^{-1}\sum_{i=1}^n \frac{\partial \log p_{\thTT}(\bm z_i) }{\partial \th}
=n^{-1}\sum_{i=1}^n \frac{\partial \log p_{\thTT}(\bm z_i) }{\partial \th} - n^{-1}\sum_{i=1}^n \frac{\partial \log p_{\thE}(\bm z_i) }{\partial \th}  
\approx n^{-1}\sum_{i=1}^n \frac{\partial^2 \log p_{\thTT}(\bm z_i) }{\partial \th^2} (\thTT-\thE) 
 $$
and the asymptotic normality of $\sqrt{n}(\thE - \thTT)$, we may further approximate $\lambda_n$ by
$ 
	(\thE-\thTT)^\T V (\thE-\thTT)  
	\approx n^{-1} E(\bm \zeta^\T V \bm \zeta ) = n^{-1} \tr(V^{-1} J),
$ 
where $\bm \zeta \sim \mathcal{N}(0, V^{-1}JV^{-1})$.
For a well-specified model, we have $V=J$ and $\lambda_n \approx n^{-1} d$ with $d$ denoting the model dimension, and thus TIC becomes AIC.

Why should TIC be preferred over AIC in nonlinear models in general? Intuitively speaking, TIC has the potential of exploiting the nonlinearity while AIC does not. Recall our Example~\ref{example:NN} in the introduction, with the loss being the negative log-likelihood. It is well known from machine learning practice that neural network structures play a key role in effective prediction. However, information criteria such as AIC impose the same amount of penalty as long as the number of neurons remains the same, regardless of how neurons are configured. 

In this paper, we extend the scope of allowable loss functions and theoretically justify the use of GTIC (and thus TIC). 
Under some regularity conditions (elaborated in the Appendix), we shall prove that 
	the $\hat{\mo}_n$ selected by the GTIC procedure is asymptotically efficient (in the sense of Definition~\ref{def:efficiency}). 
	This is formally stated as a theorem in Subsection~\ref{subsec:theorem1}.
Our theoretical results extend some existing statistical theories on AIC for linear models. 
We note that the technical analysis of high dimensional (non) linear model classes is highly nontrivial. We will develop some new technical tools in the Appendix, which may be interesting on their own rights. 
	 

\section{{\co Adaptivity to the Oracle}} \label{sec:LOL}

\subsection{Notation}
Let $\Mo$, $\mo$, $\d$, $\H[\mo] \subset \Real^{\d}$ denote respectively a set of finitely many candidate models (also called the model class), a candidate parametric model, its dimension, its associated parameter space. 
Let $\p \de \max_{\mo \in \Mo} \d$ denote the dimension of the largest candidate model. 
We will frequently use a subscript $n$ to emphasize the dependency on the sample size and include an $\mo$ in the arguments of many variables or functions to emphasize their dependency on the model (and parameter space) under consideration. 
%
For a measurable function $f(\cdot)$, we define $E_n f(\cdot) = n^{-1} \sum_{i=1}^n f(\bm z_i)$. For example, 
$
E_n \l(\cdot, \bm \th; \mo) = n^{-1} \sum_{i=1}^n \l(\bm z_i, \th; \mo).
$ 
We let $\bm \psi_{n}(\bm z, \th; \mo) \de \nabla_{\th} \l(\bm z, \th; \mo)$, and $\nabla_{\th} \bm \psi_{n}(\bm z, \th; \mo) \de \nabla^2_{\th} \l(\bm z, \th; \mo)$, which are respectively measurable vector-valued and matrix-valued functions of $\th$.
We define the matrices
\begin{align*}
	V_n(\th; \mo) 
	&\de \E \nabla_{\th} \bm \psi_{n}(\cdot, \th; \mo) \\
	J_n(\th; \mo) 
	&\de \E \big\{ \bm \psi_{n}(\cdot, \th; \mo) \times \bm \psi_{n}(\cdot, \th; \mo)^\T \bigr\}  
\end{align*}
Recall the definition of $\L[\mo]$. Its sample analog (also referred to as the \textit{in-sample loss}) is defined by 
$
	\Lhat[\mo] 
	\de E_n \l \bigl(\cdot, \thE[\mo] ; \mo\bigr)	. 
$
Similarly, we define 
\begin{align*}
	\hat{V}_n(\th; \mo) 
	&\de E_n \nabla_{\th} \bm \psi_{n}(\cdot, \th; \mo) \\
	\hat{J}_n(\th; \mo) &\de E_n \big\{ \bm \psi_{n}(\cdot, \th; \mo) \times \bm \psi_{n}(\cdot, \th; \mo)^\T \bigr\}  
\end{align*}
{\co When $\ell$ is the negative log-likelihood, the above $\psi_n$ is the score function, and $\widehat{V}_n$
and $\widehat{J}_n$ are candidates for estimating a Fisher information matrix.}


Throughout the paper, the vectors are arranged in a column and marked in bold.  
Let $\norm{\cdot}$ denote the Euclidean norm of a vector or spectral norm of a matrix.
Let $\INT(S)$ denote the interior of a set $S$. 
For any vector $\bm c \in \mathbb{R}^d$ ($d \in \N$) and scalar $r > 0$, let $B(\bm c,r)\de \{\bm x \in \mathbb{R}^d: \norm{\bm x-\bm c} \leq r \}$. 
For a positive semidefinite matrix $V$ and a vector $\bm x$ of the same dimension, we shall abbreviate $\bm x^\T V \bm x$ as $\norm{\bm x}_V^2$. 
For a given probability measure $\P$ and a measurable function $m$, let $\norm{m}_{\P} \de (\E m^2)^{1/2}$ denote the $L_2(\P)$-norm.
Unless otherwise stated, $\E $ denotes the expectation with respect to the true data generating process. 
Let $\textrm{eig}_{\min}(V)$ (resp. $\textrm{eig}_{\max}(V)$) denote the smallest (resp. maximal) eigenvalue of a symmetric matrix $V$. 
For a sequence of scalar random variables $f_n$, 
	we write $f_n=o_p(1)$ if $\lim_{n \rightarrow \infty} f_n= 0$ in probability, 
	and $f_n = O_p(1)$, if it is stochastically bounded. 
For a fixed measurable vector-valued function $\bm f$, we define
$$ \G \bm f \de \sqrt{n}(E_n - \E) \bm f , $$
the empirical process evaluated at $\bm f$. 
For $a,b \in \Real$, we write $ a \lesssim b$ if $a \leq c b$ for a universal constant $c$.
For a vector $\bm a$ or a vector-valued function $\bm f$, we let $a_i$ or $f_i$ denote the $i$th component.

We use $\rightarrow$ and $\limp$ to respectively denote the deterministic and in probability convergences. Unless stated explicitly, all the limits throughout the paper are with respect to $n \rightarrow \infty$ where $n$ is the sample size.

\subsection{Approaching the Oracle -- Selection Procedure} \label{subsec:GTIC}

An appropriate model selection procedure is necessary to strike a balance between the model fitting and {model complexity} based on the observed data to obtain the optimal predictive power.  
The basic idea of penalized selection is to impose an additive penalty term on the in-sample loss so that larger models are more penalized.
%
In this paper, we follow the aphorism that ``all models are wrong'', and assume that the model class under consideration is misspecified.

\begin{definition}[Efficient learning] \label{def:efficiency} 
	Our goal is to select $\hat{\mo}_n \in \Mo$ that is asymptotically efficient, in the sense that 
	\begin{align}
		\frac{ \L[\hat{\mo}_n] }{ \min_{\mo \in \Mo} \L[\mo] } \limp 1 \label{eq:efficiency}	
	\end{align}
	as $n \rightarrow \infty$. 
\end{definition}
Note that this requirement is weaker than selecting the exact optimal model $\argmin_{\mo \in \Mo} \L[\mo]$. 
Also, the concept of asymptotic efficiency in model selection is reminiscent of its counterpart in parameter estimation theory.  
A similar definition has been adopted in the study of the optimality of AIC in autoregressive order selection~\cite{shibata1980asymptotically} and variable selection in linear regression models~\cite{shibata1981optimal}. 

It is worth noting that the above definition is in the scope of the available data and a specified class of models. Because we are in a data-driven setting where it is unrealistic to compete with the best performance attainable with full knowledge of the underlying distribution, 
we chose the above rationale of efficient learning instead of using 
	\begin{align}
		\frac{ \L[\hat{\mo}_n] }{ \min_{\mo \in \Mo} \E \l \bigl(\cdot, \thT[\mo]; \mo \bigr) } \limp 1  \nonumber
	\end{align}
whose denominator does not reveal the influence of finite-sample data. 
In other words, Definition~\ref{def:efficiency} calls for a model whose predictive power can practically approach the best offered by the candidate models (i.e., the oracle in Definition~\ref{def:lol}). 

A related but different school of thoughts is {structural risk minimization} in the statistical learning literature. In that context, the out-sample prediction loss is usually bounded using in-sample loss plus a positive term (e.g., a function of the Vapnik-Chervonenkis (VC) dimension~\cite{vapnik1971uniform} for a classification model). 
Definitive treatment of this line of work can be found in, e.g.,~\cite{shawe1998structural,koltchinskii2001rademacher,bousquet2003introduction,koltchinskii2011introduction} and the references therein. The major difference of our setting compared with that in learning theory is our requirement that the positive term plus the in-sample loss should asymptotically approach the true out-sample loss (as sample size goes to infinity).  
 
Another related notion often used to describe model selection performance is {minimax-rate optimality}~\cite{barron1999risk,yang2005can}.
In nonparametric estimation of the regression function $f\in \mathcal{F}$ under the squared loss, tight minimax risk bounds for $\inf_{\hat{f}} \sup_{f \in \mathcal{F}} n^{-1} \sum_{i=1}^n \E (\hat{f}(x_i)-f(x_i))^2$ have been obtained since the pioneering work of \cite{pinsker1980optimal,pinsker1984learning} (see~\cite{nussbaum1999minimax,tsybakov2008introduction} for more discussions).
A model selection method $\nu$ is said to be minimax-rate optimal over $\mathcal{F}$, if $\sup_{f \in \mathcal{F}} n^{-1} \sum_{i=1}^n \E \{\hat{f}_{\nu}(x_i)-f(x_i)\}^2$ converges at the same rate as the aforementioned minimax risk, where $\hat{f}_{\nu}$ is the least squares estimate of $f$ under the variables selected by $\nu$. 
 In contrast to the notion of asymptotic efficiency, which we focus on here, minimax-rate optimality allows the true data-generating model to vary and thus is a stronger requirement. The asymptotic efficiency is in a pointwise sense, meaning that a fixed but unknown data-generating process already generates the data.
It has been proved that AIC is minimax-rate optimal for a range of variable selection tasks, and there exists no model selection method that achieves such optimality as well as selection consistency~\cite{yang2005can}.
Meanwhile, it is possible to simultaneously combine asymptotic efficiency and selection consistency, and that motivated recent research in reconciling AIC-type and BIC-type model selection methods~\cite{ing2007accumulated,erven2012catching,zhang2015cross,DingBridge}.

We propose to use the following penalized model selection procedure, which extends TIC from negative log-likelihood to general loss functions. 

\textbf{Generalized TIC (GTIC) procedure:} Given data $\bm z_1,\ldots, \bm z_n$ and a specified model class $\Mo$. We select a model $\hat{\mo} \in \Mo$ in the following way:
1) for each $ \mo \in \Mo $, find the minimal loss estimator $\thE[\mo]$ defined in (\ref{eq:MLE}), and record the minimum as $\Lhat[\mo]$;  
2) select $\hat{\mo} = \argmin_{\mo \in \Mo} \Lcn[\mo]$, where  
	\begin{align}
		 \Lcn[\mo] \de 
		\Lhat[\mo] + n^{-1} \tr\bigl\{ \hat{V}_n(\thE[\mo]; \mo)^{-1} \hat{J}_n(\thE[\mo]; \mo)\bigr\} . \label{eq30} 
	\end{align}

We note that the two additive terms on the right-hand side of (\ref{eq30}) represent the fitting performance and the model complexity, respectively.  

The quantity $\Lcn[\mo]$, also referred to as the corrected prediction loss, can be calculated from data. It serves as a surrogate for the out-sample prediction loss $\L[\mo]$, which is usually not computable. The in-sample loss $\Lhat[\mo]$ cannot be directly used as an approximation for $\L[\mo]$, because it uses the sample approximation twice: once in the estimation of $\thT$, and then in the approximation of $\E \l(\cdot,\th;\mo)$ using $E_n \l(\cdot,\th;\mo)$ (the law of large numbers).  
For example, in a nested model class, the largest model always has the least $\Lhat[\mo]$ (i.e., fits data the best). But as we discussed in the introduction, $\L[\mo]$ is typically decreasing first and then increasing as the dimension increases. 

\subsection{Asymptotic Analysis of the GTIC Procedure} \label{subsec:theorem1}

We need the following assumptions for asymptotic analysis. 

\begin{assumption} \label{ass:data}
	Data $\bm Z_i, i=1,\ldots,n$ are independent and identically distributed (i.i.d.). 	
\end{assumption}

Assumption~\ref{ass:data} is standard for theoretical analysis and some practical applications. In the context of regression analysis, it corresponds to the random design.  
In our technical proofs, it is possible to extend the assumption of i.i.d. to strong mixing~\cite{bradley1986basic}, which is more commonly assumed for time series data.

\begin{assumption}\label{ass:unique}
	For each model $\mo \in \Mo$, $\thT[\mo]$ (as was defined in (\ref{eq:tMLE})) is in the interior of the compact parameter space $\H[\mo]$, and for all $\v>0$ we have 
	\begin{align}
		&\liminf_{n \rightarrow \infty} \inf_{\mo \in \Mo} \biggl( 
		\inf_{\th \in \H[\mo]: \norm{\th-\thT[\mo]} \geq \v } \E \ell \bigl(\cdot, \th ; \mo \bigr)  - \E \ell \bigl(\cdot, \thT[\mo] ; \mo \bigr) \biggr)  
		\geq \eta_{\v} \nonumber
	\end{align}
	for some constant $\eta_{\v} > 0$ that depends only on $\v$. Moreover, we have 
	\begin{align}
		&\sup_{\mo \in \Mo}\sup_{\th \in \H[\mo]} \biggl| E_n \ell \bigl(\cdot, \th ; \mo \bigr) - \E \ell \bigl(\cdot, \th ; \mo \bigr) \biggr| \limp 0 , \nonumber 
	\end{align}
	as $n\rightarrow \infty$, and $\ell (\cdot, \thT[\mo] ; \mo )$ is twice differentiable in $\INT(\D)$ for all $n$, $\mo \in \Mo$. 
\end{assumption}

Assumption~\ref{ass:unique} is the counterpart of the separated mode and uniform law of large number conditions that have been commonly required in proving the consistency of maximum likelihood estimator for classical statistical models (see, e.g.,~\cite[Theorem 5.7]{van2000asymptotic}). The $\thT[\mo]$ can be interpreted as the oracle optimum under model $\mo$, or a ``projection'' point of the true data generating distribution onto the model $\mo$.  


\begin{assumption}\label{ass:converge} 
	There exist constants $\tau \in (0,0.5)$ and $\delta>0$ such that
	\begin{align}
		&\sup_{\mo\in \Mo}\sup_{\th \in \H[\mo] \cap B(\thT[\mo],\delta)} n^{\tau} \norm{E_n \bm \psi_{n}(\cdot, \th ; \mo) - \E \bm \psi_{n}(\cdot, \th ; \mo)} = O_p(1) \nonumber	. 
	\end{align}
	Additionally, the map $\th \mapsto \E \bm \psi_{n}(\cdot, \th; \mo)$ is differentiable at $\th \in \INT(\H[\mo])$ for all $n$ and $\mo \in \Mo$.
\end{assumption}

Assumption~\ref{ass:converge} is a weaker statement compared with the central limit theorem and its extension to Donsker classes in a classical (non-high dimensional) setting. In our high dimensional setting, the assumption ensures that each projected model $\thT[\mo]$ behaves regularly. It implicitly builds a relation between $\p$, the dimension of the largest candidate models, and sample size $n$. 
As was pointed out by an anonymous reviewer, it is technically possible to replace $n^\tau$ with a weaker requirement, say $n^{0.5}/(\log n)^\gamma$ for any constant $\gamma>0$.

\begin{assumption}\label{ass:eig} 
	There exist constants $c_1,c_2 > 0$ such that 
	\begin{align*}
	&\liminf_{n \rightarrow \infty} \min_{\mo \in \Mo} \textrm{eig}_{\min}(V_n(\thT ; \mo)) \geq c_1, \\
	&\limsup_{n \rightarrow \infty} \max_{\mo \in \Mo} \textrm{eig}_{\max}(V_n(\thT ; \mo)) \leq c_2 . 
	\end{align*} 
\end{assumption}


Assumption~\ref{ass:eig} assumes that the second derivative of the out-sample prediction loss has bounded eigenvalues at the optimum $\thT[\mo]$. 
The lower bound $c_1$ 
indicates that the loss function is strongly convex for all models, and the upper bound $c_2$ requires the loss functions to be reasonably smooth.
This assumption is used in our asymptotic analysis to ensure reasonable Taylor expansions up to the second order.

 
\begin{assumption}\label{ass:Lip}
	There exist fixed constants $r>0$, $\gamma > 1$, and measurable functions $m_n[\mo]: \D \rightarrow \Real^{+}\cup\{0\}$, $\bm z \mapsto m_n[\mo](\bm z)$ for each $\mo \in \Mo$, such that for all $n$ and $\th_1,\th_2 \in B(\thT[\mo],r)$, 
	\begin{align}
		&\norm{\bm \psi_{n}(\bm z, \th_1; \mo)-\bm \psi_{n}(\bm z, \th_2; \mo)} \leq m_n[\mo](\bm z) 	\norm{\th_1-\th_2}, \label{eq:Lip} \\ 
		&\E m_n[\mo] < \infty . \label{eq:exchangeable}
	\end{align}
	Moreover, we have 
	\begin{align}
		&\max\biggl\{ {\co \p^{2\gamma} \ \card(\Mo)^{\gamma} } ,\ \p \sqrt{\log \{ \p \card(\Mo) \}} \biggr\}\ 
		\times n^{-\tau} \biggl\lVert \sup_{\mo \in \Mo} m_n[\mo] \biggr\rVert_{\P} \rightarrow 0 .\label{eq:mCond} 
	\end{align}  

\end{assumption}

Assumption~\ref{ass:Lip} is a Lipschitz-type condition. Similar but simpler forms of this have been used in classical analysis of asymptotic normality \cite[Theorem 5.21]{van2000asymptotic}. We note that the condition (\ref{eq:mCond}) explicitly requires that the largest dimension $\p$ and the candidate size $\card(\Mo)$ do not grow too fast as $n$ goes to infinity. The condition (\ref{eq:mCond}) is used to bound the rate of convergence of the empirical process $\G \bm \psi_{n}(\cdot, \th ; \mo)$ in the vicinity of $\th=\thE[\mo]$. Similar conditions were often used to establish asymptotic results such as the Cram\'{e}r-Rao bound~\cite[Theorem 18]{ferguson2017course}.

%
%
%
%


\begin{assumption} \label{ass:secondOrderLLN}
There exists a constant $\delta>0$ such that
	\begin{align}
		&\sup_{\mo \in \Mo}\sup_{\th \in \H[\mo] \cap B(\thT[\mo],\delta)} \norm{\hat{J}_n(\th ; \mo) - J_n(\th ; \mo) } \limp 0,  
			\label{eq:consistJ} \\
		&\sup_{\mo \in \Mo}\sup_{\th \in \H[\mo] \cap B(\thT[\mo],\delta)} \norm{\hat{V}_n(\th ; \mo) - V_n(\th ; \mo) } \limp 0,  
			\label{eq:consistV} \\
		&\lim_{\v \rightarrow 0} \sup_{\mo \in \Mo} \sup_{\th \in \H[\mo] \cap B(\thT[\mo],\v)} \norm{V_n(\th ; \mo) - V_n(\thT ; \mo) } = 0. \label{eq:uniformV}
	\end{align}
\end{assumption}

Assumption~\ref{ass:secondOrderLLN} requires that the sample analogs of the matrices $J_n(\th ; \mo)$ and $V_n(\th ; \mo)$ are asymptotically close to the truth (in spectral norm) in a neighborhood of $\thT[\mo]$. In the classical setting, it is guaranteed by the law of large numbers (applied to each matrix element). 
{\co The above uniform convergence conditions may be further simplified using finite sample properties of random covariance-type matrices, e.g., a recent result in \cite{oliveira2013lower}.}
Assumption~\ref{ass:secondOrderLLN} also requires the continuity of $V_n(\th ; \mo)$ in a neighborhood of $\thT[\mo]$. 

We define  
	$$\bm w_n[\mo] = \frac{1}{\sqrt{n}} \sum_{i=1}^n \bm \psi_n(\bm z_i,\thT[\mo] ; \mo). $$ 
	Clearly, $\bm w_n[\mo]$ has zero mean and variance matrix $J_n(\thT[\mo]; \mo)$, and thus 
	\begin{align*}
		\E \norm{\bm w_n[\mo]}_{V_n(\thT[\mo]; \mo)^{-1}}^2 = \tr \bigl\{ V_n(\thT[\mo]; \mo)^{-1} J_n(\thT[\mo]; \mo) \bigr\}.	
	\end{align*}

\begin{assumption} \label{ass:risk}

Suppose that the following regularity conditions are satisfied.  	
\begin{align}
	 &\inf_{\mo \in \Mo} n^{2\tau} \R[\mo] \rightarrow \infty , \label{eq84} \\
	 &\sup_{\mo \in \Mo} \frac{\d}{n \R[\mo]} \rightarrow 0 .\label{eq93}
\end{align}
Moreover, 
there exists a fixed constant $m_1 > 0 $ such that 
\begin{align}
	\sum_{\mo \in \Mo}
	(n\R[\mo])^{-2m_1} &\E \bigl\{\l(\cdot, \thT[\mo] ; \mo)- 
	\E \l(\cdot, \thT[\mo]; \mo) \bigr\}^{2m_1} 
	 \rightarrow 0 , \label{eq85} 
\end{align}
there exists a fixed constant $m_2 > 0 $ such that 
\begin{align}
	&\sum_{\mo \in \Mo} 
	(n\R[\mo])^{-2m_2} \E\biggl[ \norm{\bm w_n[\mo]}_{V_n(\thT[\mo]; \mo)^{-1}}^2 
	- \tr \bigl\{ V_n(\thT[\mo]; \mo)^{-1} J_n(\thT[\mo]; \mo) \bigr\}\biggr]^{2m_2} \rightarrow 0, \label{eq89}
\end{align}
and there exists a fixed constant $m_3 > 0 $ such that 
\begin{align}
	\limsup_{n \rightarrow \infty} 
	\sum_{\mo \in \Mo} (n\R[\mo])^{-m_3} \{ 
	&\E \norm{\bm w_n[\mo]}^{m_3} + 
	\E \norm{\bm w_n[\mo]}^{2 m_3}	\} < \infty . \label{eq88}
\end{align}
 
\end{assumption}

In Assumption~\ref{ass:risk}, the conditions (\ref{eq84}), (\ref{eq93}) and (\ref{eq88}) indicate that the risks $\R[\mo]$ for all $\mo$ are not small so that the model class is virtually mis-specified. 
The assumptions of (\ref{eq85}) and (\ref{eq89}) are central moment constraints that control the regularity of loss functions. Similar conditions were often used to establish the asymptotic performance of model selection, for example \cite[Condition (2.6)]{shao1997asymptotic} and \cite[Condition (A.3)]{li1987asymptotic}. 

{\co Overall, Assumptions~1-7 ensure that the conditions for asymptotic normality in regular parametric models are supplemented with conditions ensuring a sufficient level of uniformity among models. }

%
%
%

\begin{theorem}\label{theorem:efficiency}
	Suppose that Assumptions~1-7 hold.
	Then the $\hat{\mo}_n$ selected by GTIC procedure is asymptotically efficient (in the sense of Definition~\ref{def:efficiency}).
\end{theorem}


Classical asymptotic analysis for general parametric models with i.i.d. observations typically relies on a type of uniform convergence of empirical process around $\thT[\mo]$ within a fixed parameter space. 
Because our functions are vector-valued with dimension depending on the sample size $n$, we cannot directly use state-of-the-art technical tools such as those in \cite[Theorem 19.28]{van2000asymptotic}. 
The classical proof by White~\cite{white1982maximum} (in proving asymptotic normality in misspecified class) cannot be directly adapted, either, for parameter spaces that depend on $n$. 
On the other hand, though asymptotic analysis for criteria such as AIC, $C_p$, CV, GIC often consider models that depend on $n$ (see, e.g.,~\cite{shibata1981optimal,li1987asymptotic,shao1997asymptotic,yang2005can}, it is often studied in the context of fixed-design regression models, so the technical tools there cannot be directly applied for our purpose. 

Some new technical tools are needed in our proof. 
Here we sketch some technical ideas in the proof.
We first prove that $\thE[\mo]$ is $n^\tau$-consistent (instead of the classical $\sqrt{n}$-consistency).
We then prove the first key result, namely Lemma~\ref{lemma:locUnifConv}, that states a type of local uniform convergence. Note that its proof is nontrivial as both the empirical process and $\thE$ depend on the same observed data. Our technical tools resemble those for proving a Donsker class, but the major difference is that our model dimensions depend on $n$. 
We then prove the second key lemma, Lemma~\ref{lemma:normality}. It directly leads to the asymptotic normality of maximum likelihood estimators in the classical setting. It is somewhat interesting to see that the proof of Lemma~\ref{lemma:normality} does not require the $\sqrt{n}$-consistency of $\thE[\mo]$, which usually does not hold in high dimensional settings. 
 

\subsection{Example}

Theorem~\ref{theorem:efficiency} applies to general parametric model classes, where assumptions can often be simplified. 
We shall use regression models as an example of applying Theorem~\ref{theorem:efficiency}. 
Suppose that the response variable is written as $Y = \mu(\bm X) + \v$, where $\v$ is a random noise with mean zero and variance $\sigma^2$, and $\mu(\bm X)$ is a possibly nonlinear function of $\p$ predictors $\bm X=[X_1, \ldots, X_{\p}]^\T$. In linear models, data analysts assume that $\mu$ is a linear function of $\bm X$ in the form of $\mu = \beta_1 X_1 + \cdots  +\beta_{\p} X_{\p}$, where $\p$ may or may not depend on the sample size $n$. We sometimes write $\mu(\bm X)$ as $\mu$ for brevity. 
For simplicity, we assume that $\sigma$ is known, and $\bm X$ is a random vector independent with $\v$. Also assume that $E(Y)=0$ and $E(X_i)=0$ ($i=1,\ldots,\p$).
The observed data are $n$ independent realizations of $Z=(Y,X_1,\ldots,X_{\p})$. The unknown parameters are $\bm \theta = (\beta_1,\ldots,\beta_{\p})$. The model class, denoted by $\Mo$, consists of candidate models represented by $\mo \subseteq \{1,\ldots,\p\}$, i.e. $\mu(\bm X) = \sum_{i \in \mo} \beta_i X_i$. 

In regression, it is common to use the quadratic loss function
  \begin{align}
    \l(\bm z, \th; \mo) = \biggl(y - \sum_{j \in \mo} \beta_j x_j\biggr)^2 - \sigma^2 \nonumber
  \end{align} 
  for $\th \in \H[\mo]$. {\co The subtraction of $\sigma^2$ allows for better comparison of competing models.}
  Note that the population loss is 
  \begin{align}
    \E \l(\bm z, \th; \mo) = \E \biggl(\mu - \sum_{j \in \mo} \beta_j x_j\biggr)^2 . \label{e1}
  \end{align}
Suppose that $\thT$ is defined as in (\ref{eq:tMLE}). We define $\S_{xx}$ to be the covariance matrix whose $(i,j)$-th element is $\E (X_i X_j)$,  $\S_{x\mu}$ to be the column vector whose $i$-th element is $\E (X_i \mu)$, and $\S_{\mu \mu} = \E(\mu^2)$. 
  We similarly define $\S_{xx}[\mo]$, $\S_{x\mu}[\mo]$, $\bm X[\mo]$ which are the covariance matrix/vectors restricted to model $\mo \in \Mo$.   
  Simple calculations show that $\thT[\mo] = (\S_{xx}[\mo])^{-1} \S_{x\mu}[\mo]$ for $\H(\mo) = \mathbb{R}^{\d}$, and (\ref{e1}) may be rewritten as 
  \begin{align}
    \E \l(\bm z, \th; \mo) 
    &= \E \l(\bm z, \thT[\mo]; \mo) + \norm{\th-\thT[\mo]}_{\S_{xx}[\mo]}^2 \nonumber \\
    &= \bigl( \S_{\mu \mu}  - \S_{\mu x}[\mo] \S_{xx}[\mo]^{-1} \S_{x \mu}[\mo] \bigr) + 
     \norm{\th-\thT[\mo]}_{\S_{xx}[\mo]}^2. \label{eq:tradeoff}  
  \end{align}
  The decomposition in (\ref{eq:tradeoff}) has a nice interpretation in terms of bias-variance tradeoff. The first term is the $L_2(\P)$-norm of the orthogonal complement of $\mu$ projected to the linear span of covariates, or the  minimal possible loss offered by the specified model $\mo$. Clearly, it is zero if $\mo$ is well-specified, and nonzero otherwise. The second term represents the variance of estimation.  
Evaluating $\hat{V}_n(\th; \mo)$ and $\hat{J}_n(\th; \mo)$ in this specific case, we obtain 
\begin{align*}
	\hat{V}_n(\th; \mo) = 2 E_n (\bm X[\mo]  \bm X[\mo]^\T), \quad
	\hat{J}_n(\th; \mo)	 = 4  E_n (\bm X[\mo]  \bm X[\mo]^\T R[\mo]^2 ),\quad
	R[\mo] \de Y - \thE[\mo]^\T \bm X[\mo] .
\end{align*}
Note that when $R[\mo] $ is close to the independent noise term, then $\E(R[\mo]^2 ) \approx \sigma^2$ and the GTIC penalty in (\ref{eq30}) is around $n^{-1} (2 d \sigma^2)$ which approximates the AIC and Mallows' $C_p$ method.
Theorem~\ref{theorem:efficiency} implies the following corollary. 
In verifying the previous assumptions such as Assumption~\ref{ass:unique} for this corollary, we used the fact that $\norm{\thT[\mo]}=  \norm{ (\S_{xx}[\mo])^{-1} \S_{x\mu}[\mo] } = \norm{\S_{x\mu}[\mo] }\leq c^2 \p$, and the least squares estimates fall into $\{\bm \th : \norm{\bm \th} \leq 2c \sqrt{\p}\}$ with high probability (due to the concentration inequalities for bounded $X_i$ and $\mu$). 
It is possible to relax the conditions by a more sophisticated verification of assumptions.

\begin{corollary}\label{coro:linear}
	Assume that $|\mu|$ and $|X_i|$ ($i=1,\ldots,\p$) are bounded by a constant $c$ that does not depend on $n$. Suppose the following conditions hold, then the $\hat{\mo}_n$ selected by GTIC procedure is asymptotically efficient.\\
	1) $X_1,\ldots,X_{\p}$ are independent with zero mean and unit variance for all $n$; \\
	2) $\p = o( n^{w})$, where $w<1/8$; \\
	3) $\inf_{\mo\in \Mo}\R[\mo] > n^{-\zeta}$, where $\zeta < 1-2w$; \\
	4) 
	{\co $\card(\Mo) = O(n^{\nu})$, where $\nu < 1/2-4w $. }
\end{corollary}

\section{Sequential Model Expansion} \label{sec:expand}

As explained in the introduction, 
in terms of predictive power, a model in a misspecified model class could be determined to be unnecessarily large, suitable, or inadequately small, depending on a specific sample size (see Fig.~\ref{fig_intro}). 
A realistic learning procedure thus requires models of different complexity levels as more data become available.

Throughout this section, we shall use $T$ (instead of the previously used $n$) to denote sample size, and subscript $t$ as the data index, in order to emphasize the sequential setting. 
 

%

\subsection{Discussion}
We have addressed the selection of an efficient model for a given number of observations.
In many practical situations, data are sequentially observed. A straightforward model selection is to repeatedly apply the GTIC procedure upon arrival of data. 
However, in a sequential setting, the following issue naturally arises:


\textit{Suppose that we successively select a model and use it to predict at each time step. The path of the historically selected models may fluctuate a lot. Instead, it is more appealing (either statistically or computationally) to force the selected models to evolve gradually. 
}

To address the above challenge, we first propose a concept referred to as the \textit{graph-based} expert tracking, which extends some classical online learning techniques (Algorithm~\ref{algo:trackGraphOrigin}). 
Motivated by the particular path graph $1 \rightarrow 2 \rightarrow \cdots N$, where $1,2,\ldots,N$ index the candidate models, we further propose a model expansion strategy (Algorithm~\ref{algo:trackExpert}), where each candidate model and its corrected prediction loss can be regarded respectively as an expert and loss.  

The proposed algorithm can be used for online prediction, which ensures not only statistically reliable results but also simple computation.  
Specifically, we propose a predictor that has cumulative out-sample prediction loss (over time) close to the following optimum benchmark:
\begin{align}
	\min_{\size(i_1,\ldots,i_T)\leq k, \ i_1,\ldots, i_T \in \{1,\ldots,N\}} \sum_{t=1}^T \L[\mo_{i_t}] .
\end{align}
where the size of a sequence $\size(i_1,\ldots,i_T)$ is defined as the number of $t$'s such that $i_t \neq i_{t+1}$.
In other words, the minimization is taken over all tuples $(i_1,\ldots,i_T)$ that have at most $k$ switches and that are restricted to the chain $1 \rightarrow 2 \rightarrow \cdots$. For example, $(i_1,\ldots,i_5)=(1,2,2,3,3)$. 
In the above formulation, $i_t$ and $k$ respectively mean the index of the model chosen to predict at time step $t$, and the number of switches within $T$ time steps.

\subsection{Tracking the Best Expert with Graphical Constraints}

In this subsection, we propose a novel graph-based expert tracking technique that motivates our algorithm in the following subsection.
The discussion may be interesting on its own right, as it includes the state-of-art expert tracking framework as a special case (when the underlying graph is fully-connected/complete). 

Suppose there are $N$ experts. At each discrete time step $t=1,2,\ldots,T$, each expert gives its prediction, after which the environment reveals the truth $\yt\in\yy$. In this subsection, with a slight abuse of notation, we shall also use $\ls$ to denote loss functions in the context of online learning. The performance of each prediction is measured by a loss function $\ls: \{1,2,\ldots,N\}\times\yy\rightarrow\RR$.  
A smaller loss indicates a better prediction. 
In light of the model expansion we shall introduce in the next subsection, each $i=1,\ldots,N$ represents a model, and $\ls(i,\bm z_t)$ is the prediction loss of model $i$ which is successively re-estimated using $\bm z_1,\ldots,\bm z_{t}$ at time step $t$.  

In order to aggregate all the predictions that the experts make, we maintain a weight value for each expert and update them upon the arrival of each new data point based on the qualities of the predictions. We denote the weight for expert $i\in\{1,\ldots,N\}$ at time $t$ as $w_{i,t}$, and the normalized version as $W_{i,t}$. 
The goal is to optimally update the weights for a better prediction, which is measured by the cumulative loss minus the best achievable (benchmark) loss. 
This measure is often called ``regret'' in the online learning literature~\cite{stoltz2005internal,cesa2006prediction,stoltz2007learning}. 
The regret is a relevant criterion of evaluating the predictive performance in sequential settings since the model and model parameters have to be adjusted on a rolling basis as new data arrives, and a selected model at a time step $t_1$ may not be suitable at another time step $t_2$. 
If the benchmark in the regret is defined as the minimum cumulative loss achieved by a single expert in hindsight, namely
$ 
	\min_{1\leq i^* \leq N} \sum_{t=1}^T l(i^*,\yt) 
$, 
then it is standard to apply the exponential re-weighting procedure which produces some desirable regret bound~\cite[Chapter 2]{cesa2006prediction}. In many cases the best performing expert can be different from one time segment to another, motivating the benchmark 
$$ 
	\min_{\size(i_1,\ldots,i_T)\leq k, \, i_1,\ldots,i_T \in\{1,\ldots, N\} } \sum_{t=1}^T l(i_t,\yt) 
$$ 
where $k$ denotes the maximum number of switches of the best experts in hindsight. 
In this scenario, the fixed share algorithm~\cite[Chapter 5]{cesa2006prediction} can be a good solution with guaranteed regret bound. We consider the following problem setting that aims to significantly reduce computational costs. 

The best performing expert is restricted to switch according to a \textit{directed graph}, $G=(V,E)$ (without self-loops), with $V=\{1,\ldots,N\}$ denoting the set of nodes (representing experts) and $E$ denoting the set of directed edges. At each time point, the best performing expert can either stay the same or jump to another node which is directly connected from the current node. Let 
\begin{align}
	\bt_{ij}=\ind{\exists(i,j)\in E}, \label{eq101}
\end{align}
which is $1$ if there is a directed edge $(i,j)$ on the graph, and $0$ otherwise. Let 
\begin{align}
	\bt_{i}=\sum_{j=1}^N \bt_{ij}, \label{eq201}
\end{align}
which is the out-degree of the node $i$. In addition, we assume that $\max_{i\in{1,\ldots,N}} \bt_i \leq \deg$, where $0<\deg < N$.

We propose Algorithm~\ref{algo:trackGraphOrigin} to follow the best expert with the graphical transitional constraints.
We use a special prior $w_{i,0}$ here to motivate content in the next subsection. It is not difficult to extend our discussion to more general priors here. 
The classical fixed-share algorithm can be seen as a special case when the graph is complete. The advantage of using the graph-based expert learning is to reduce the computational cost and to obtain a tighter error bound, as shown in the following Theorem~\ref{thm:graphexpert}.
A regret bound will be derived that only depends on the graph degree $D$ instead of the number of experts $N$.
To the best of the authors' knowledge, the framework concerning dynamic regret with graph constraints stated here has not been studied before. 

\begin{algorithm}[tb]
\vspace{-0.0 cm}
\caption{Tracking the best expert with graphical transitional constraints}
\label{algo:trackGraphOrigin}
\begin{algorithmic}[1]
\INPUT Learning rate $\eta>0$, sharing rate $0<\kappa<1/\deg$
\OUTPUT $\bm p_t=[p_{t,1},\ldots, p_{t,N}]^\T$ (predictive distribution over the active models) for each $t=1, \ldots,T$
\STATE Initialize $w_{1,0}=1$ $w_{i,0}=0$ for all $i\in\{2,\ldots,N\}$ 
\FOR {$t = 1 \to T$ } 
	\STATE Calculate the predictive distribution $p_{i,t}=w_{i,t-1}/\sum_{j=1}^N w_{j,t-1}$, for each $i\in\{1,\ldots,N\}$
	\STATE Read $\yt$, and compute $v_{i,t}=w_{i,t-1}\exp(-\eta\cdot l(i,\yt))$, for each $i\in\{1,\ldots,N\}$
	\STATE Let $w_{i,t}=\kappa \sum_{j=1}^N\bt_{ji} v_{j,t}+(1-\kappa \bt_{i}) v_{i,t}$ for each $i\in\{1,\ldots,N\}$, where $\bt_{ji}, \bt_{i}$ are defined in (\ref{eq101}), (\ref{eq201})
\ENDFOR
\end{algorithmic}
 \vspace{-0.0cm}%
\end{algorithm}

\begin{theorem}
\label{thm:graphexpert}
Suppose the loss function takes values from $[0,1]$. For all $T\geq 1$, the output of the algorithm in Algorithm~\ref{algo:trackGraphOrigin} satisfies 
\begin{align*}
&\sum_{t=1}^T \biggl(\sum_{i=1}^N l(i,\yt)p_{i,t}-l(i_t,\yt)\biggr) 
\leq \frac{1}{\eta}(T-k-1)\log\frac{1}{1-\kappa \deg}+\frac{1}{\eta}k\log\frac{1}{\kappa}+\eta\frac{T}{8}
\end{align*}
for all expert sequence $(i_1,i_2,\ldots,i_T)$ and all observation sequence $(\bm z_1, \bm z_2,\ldots,\yt)$, given that $(i_1,i_2,\ldots,i_T)$ has only transitions following directed paths in graph $G$ and $\size(i_1,i_2,\ldots,i_T)\leq k$.

\end{theorem}

The left-hand side of the above inequality is referred to as \textit{regret}. 
In order to minimize the above regret bound with respect to $(\kappa,\eta)$, we first take derivative with respect to the sharing rate $\kappa$ and solve the first-order equation to obtain $\kappa=k/\bigl((T-1)\deg\bigr)$. Then the bound becomes $S/\eta+\eta T/8$. We further minimize it over the learning rate $\eta$ to obtain $\eta=\sqrt{8S/T}$. The corresponding minimal bound is calculated to be $\sqrt{TS/2}$. 
Here  
$$S = (T-1)H(k/(T-1))+k\log \deg,$$
and $H(\cdot)$ is the binary entropy function defined by $H(x) \de -x\log x-(1-x)\log(1-x)$ for $x\in(0,1)$, $H(0)=H(1)=0$.
The $T$ is interpreted as a pre-determined stopping time when the performance of the data-driven algorithm is to be contrasted with that of the optimal graph search (with $k$ node switches). 
In particular, for small $k/T$, the average of the regret is at the order of 
\begin{align*}
	\frac{1}{T } \sum_{t=1}^T \biggl(\sum_{i=1}^N l(i,\yt)p_{i,t}-l(i_t,\yt)\biggr) 
	\leq \frac{\sqrt{TS/2}}{T}
	\sim T^{-1/2} \cdot \biggl\{T \cdot \frac{k}{T} \cdot \log T + k \log D \biggr\}^{1/2}  
	\sim (k/T)^{1/2} \log (DT)
\end{align*}

It is interesting to see that with graphical constraint, the regret bound does not depend on $N$, but on the maximum out-degree $\deg$ instead. Thus, the bound can be tight even when $N$ grows exponentially in $T$, as long as $\deg \ll N $ (i.e., sparse graph).

\subsection{Algorithm for Sequential Model Expansion}

The new online learning theory proposed in the last subsection is motivated by graph-based expert tracking.
Intuitively speaking, instead of using the exponentially updated weights directly, each expert borrows some weights from others, allowing poorly performing experts to quickly stand out when they start doing better.
In that way, the experts are encouraged to rejuvenate their past performance and ``start a new life'' so that we can track the best expert in different time epochs. 
The classical fixed-share algorithm \cite[Chapter 5]{cesa2006prediction} is a special case when $\beta_{ij}=1$ for all $i\neq j$ and $\kappa$ becomes $\kappa/(N-1)$, illustrated in Fig.~\ref{fig:directionalExpert}(a). 

Our algorithm 
in this subsection is motivated by the particular \textit{path graph} $1 \rightarrow 2 \rightarrow \cdots N$, where $1,2,\ldots,N$ index the models and the corresponding $D$ is 1.
In other words, we share the weights in a directional way, thus encouraging the experts to switch in a chain. The update rule is illustrated by Fig.~\ref{fig:directionalExpert}(b). 

Our algorithm for sequential model expansion is summarized in Algorithm~\ref{algo:trackExpert}, where each candidate model and its corrected prediction loss can be regarded respectively as an expert and loss.  
The labeling of models $\mo_1,\mo_2,\ldots$ is generally in the ascending order of their dimensions. 
To further reduce the computational cost, we maintain only an active subset (of size $K$) instead of all the candidate models at each time. The active subset starts from $\{\mo_1,\ldots,\mo_K\}$; it switches to $\{\mo_2,\ldots,\mo_{K+1}\}$ when the weight of the smallest model $\mo_1$ becomes small, and that of the largest model $\mo_K$ becomes large; it continues to switch upon the aggregation of data. 

The output of Algorithm~\ref{prop:regret} is a predictive distribution over the active models. It can be used in the following two ways in practice: 1) we randomly draw a model according to the predictive distribution and use the predictor of that model, or 2) we use the weighted average of predictors of each model according to the predictive distribution. This can be regarded as a specific ensemble learning (or model averaging) method. 
The following Proposition~\ref{prop:regret} shows that with appropriate learning parameters, the average predictive performance of our algorithm is asymptotically close to the average of a series of truly optimal models (i.e., optimal model expansion), allowing moderately many switches.

\begin{figure}[tb]
\centering
 \includegraphics[width=3in]{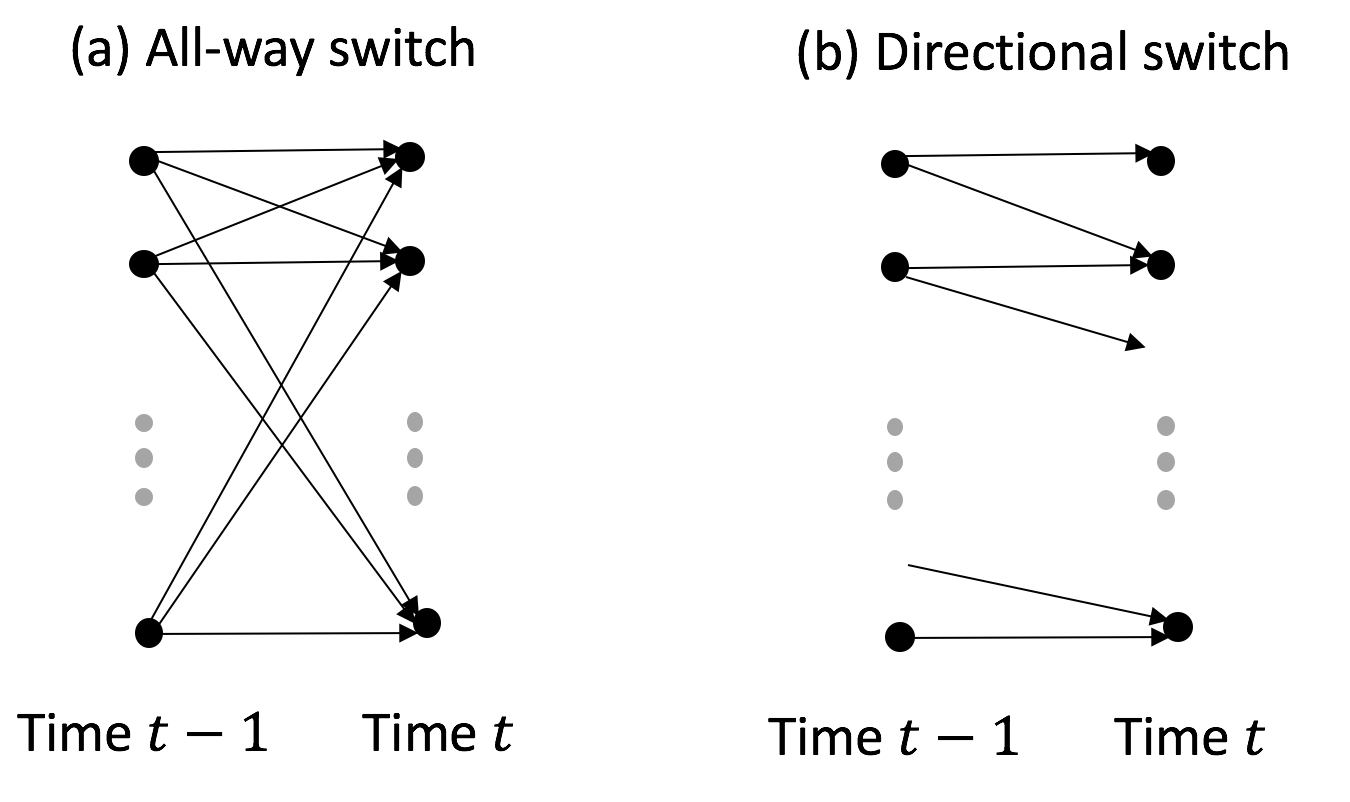}
 \vspace{-0.0in}
 \caption{Illustration of the state-of-art and our new way of redistributing the share of weights in online learning. }
 \label{fig:directionalExpert}
 \vspace{-0.0in}
\end{figure}

\begin{algorithm}[tb]
\vspace{-0.0 cm}
\small
\caption{Sequential model expansion using GTIC-corrected loss (GTIC-sequential)}
\label{algo:trackExpert}
\begin{algorithmic}[1]
\INPUT $\{ \bm z_t : t=1,\ldots ,T\}$, $\eta > 0$, $\kappa \in [0,1]$, $w_{0,1}=1, w_{0,2}=\cdots=w_{0,K}=0$, candidate models $\MoT=\{\mo_1,\mo_2,\ldots,\mo_{\card(\MoT)}\}$, $s=0$ ($\mo_{s+1},\ldots,\mo_{s+K}$ are the maintained active subsets of models), $K \in \N$, threshold $\rho \in [0,1]$
\OUTPUT $\bm p_t=[p_{t,1},\ldots, p_{t,K}]^\T$ (predictive distribution over the active models) for each $t=1, \ldots,T$
\FOR {$t = 1 \to n$ } 
	\STATE Obtain $\yt$ and compute 
	$v_{t,k} = w_{t-1,k} \exp\{-\eta \ \Lc[\mo_{s+k}] \}$ 
	for each $k=1,\ldots,K$, where $\Lc[\mo]$ is calculated from (\ref{eq30}) and fitting the data $\bm z_1, \ldots \bm z_t$ to model $\mo$.  
%
	\STATE Let 
	\begin{align}
		w_{t,k} = \left\{
		\begin{aligned} 
			(1-\kappa) \ v_{t,k} &\quad \textrm{ if } k=1 \\
			(1-\kappa) \ v_{t,k} + \kappa \ v_{t,k-1} &\quad \textrm{ if } 1<k<K \\ 
			v_{t,k} + \kappa \ v_{t,k-1} &\quad \textrm{ if } k=K
		\end{aligned} \right. \nonumber
	\end{align}
	\STATE Let 
		$p_{t,k} = (\sum_{k=1}^K w_{t,k})^{-1} w_{t,k}$, $k=1,\ldots,K$ 
	\IF{$p_{t,1} \leq \rho$ and $p_{t,K} \geq 1-\rho$ and $s+K \leq \card(\MoT)$}
		\STATE Let $s = s+1 $
		\STATE Let $w_{t,k}=w_{t,k'}$, where $k=1,\ldots,K$ and $k'=(k+1 \textrm{ mod }K)$ (relabeling the active models)
	\ENDIF
\ENDFOR
\end{algorithmic}
 \vspace{-0.0cm}%
\end{algorithm}


\begin{proposition} \label{prop:regret}
Suppose that Assumptions~1-7 hold, and that $\sup_{1 \leq t \leq T} \sup_{\mo\in \Mo} \Lc[\mo] < c$ almost surely for some fixed constant $c>0$. 
Suppose that the lines 5-8 are removed from Algorithm~\ref{algo:trackExpert}, and that $K=\card(\Mo)$, then its output satisfies 
\begin{align}
&\frac{1}{T} \biggl( \sum_{t=1}^T \sum_{i=1}^{\card(\MoT)} p_{i,t} \, \Lc[\mo_{i}] - \min_{\size(i_1,i_2,\ldots,i_T) \leq k} \sum_{t=1}^T \Lc[\mo_{i_t}] \biggr) 
\leq  \frac{c}{\sqrt{2}} \sqrt{ H\biggl(\frac{k}{T-1}\biggr) } \label{eq:final1}
\end{align}
for all $T\geq 1$,
given that 
$$
\kappa=\frac{k}{T-1}, \
\eta=\frac{1}{c}\sqrt{8 \frac{T-1}{T} H\biggl(\frac{k}{T-1}\biggr) }. 
$$

In particular, if $k=o(T)$, we have 
\begin{align}
	&\limsup_{T\rightarrow \infty} \frac{1}{T} \biggl( \sum_{t=1}^T \sum_{i=1}^{\card(\MoT)} p_{i,t} \, \Lc[\mo_{i}] 
	 - \min_{\size(i_1,i_2,\ldots,i_T) \leq k} \sum_{t=1}^T \L[\mo_{i_t}] \biggr) \leq 0 
\end{align}
almost surely. 
  
\end{proposition}



Next, we explain some details regarding Algorithm~\ref{algo:trackExpert} and Proposition~\ref{prop:regret}.
In addition to the (sequential) data and model class, other inputs to Algorithm~\ref{algo:trackExpert} are two learning parameters $\eta,\kappa$, the number of active models $K$, and the threshold $\rho$. 
The parameters $\eta$ and $\kappa$ control the rate of learning and the rate of model expansion, respectively. The number of active models $K$ is set to reduce the computation cost when the sample size is small compared with model dimensions, and the threshold $\rho$ is used to update our active models under consideration.

In particular, upon the arrival of a new data point or a set of data points, denoted by $\bm z_t$, at each time step $t$ (line 1), we update the weight of each candidate model by a Bayes-type procedure (line 2). The loss employed in the update is the corrected prediction loss, which is directly computable from the data and which serves as an approximation of the out-sample prediction loss (as was discussed in Subsection~\ref{subsec:GTIC}). The weights of each model are then updated following the path graph (line 3). 
When the weight of the smallest model becomes small, and that of the largest model becomes large, it means the current active models are inadequately small. So we drop the smallest model and include the next large model into the active set, and adjust their weights accordingly (lines 5-8). In line 7, the weight of the removed model is assigned to the newly included one, so that the sum of the weights remains the same. 
Proposition~\ref{prop:regret} states that the average predictive performance of our algorithm is asymptotically close to that of the optimal model expansion allowing $k=o(T)$ switches. For example, if only one point arrives at each time step, and the dimension of the optimal model is at the order of $T^{\delta}$ for $\delta \in (0,1)$, then the condition is trivially satisfied. 

The proof of Proposition~\ref{prop:regret} follows directly from Theorem~\ref{theorem:efficiency} and Theorem~\ref{thm:graphexpert} (with $D=1$), by using simple manipulations.
For technical convenience, Proposition~\ref{prop:regret} is only proved by removing the part of maintaining an active subset (lines 5-8). 
We maintain an active subset mainly for computational purposes, and we experimentally observed that it does not deteriorate the predictive performance. Theoretically, this is because the excluded models are often overly large or small, and their weights are thus negligible.
Next, we make some specific assumptions to illustrate the above idea. Suppose that the corrected prediction loss of model $\mo$ with dimension $d_\mo$ is approximately
$\Lc[\mo] = c_1 d_{\mo}^{-\gamma} + c_2 d_{\mo}/t$ for some positive constants $c_1,c_2,\gamma$, consisting of a bias and a variance term at each time step $t$. 
Let $w_t[\mo],v_t[\mo]$ denote the counterpart of $w_{t,k},v_{t,k}$ for each model $\mo \in \Mo$ if it were calculated, and suppose that $v_t[\mo_1] \leq \cdots \leq v_t[\mo_m]$ for a certain $m$. 
For any $1 \leq i < j \leq m$, we have 
\begin{align*}
	w_T[\mo_j] 
	&\geq (1-\kappa) v_T[\mo_j] 	=(1-\kappa) \exp\{-\eta \mathcal{L}_T^c[\mo_j]\} w_{T-1}[\mo_j] \\
	&\geq (1-\kappa)^T \exp\{-\eta \sum_{t=1}^T \Lc[\mo_j]\} 
	\sim (1-\kappa)^T \exp\{-\eta c_1 T d_{\mo_j}^{-\gamma} - \eta c_2 d_{\mo_j} \log T \} 
\end{align*}
while on the other hand
\begin{align*}
	w_T[\mo_i] 
	& 	\leq v_T[\mo_i] = w_{T-1}[\mo_i] \exp\{-\eta \mathcal{L}_T^c[\mo_i]\} \\
	& \leq \exp\{-\eta \sum_{t=1}^T \Lc[\mo_i]\} 
	\sim \exp\{-\eta c_1 T d_{\mo_i}^{-\gamma} - \eta c_2 d_{\mo_i} \log T \}
\end{align*}
For small $k$, it can be verified that $(1-\kappa)^T \sim \exp(-k)$ and
$$
w_t[\mo_i] /w_t[\mo_j] = O(1) \times \exp\{-\eta c_1 T (d_{\mo_i}^{-\gamma} - d_{\mo_j}^{-\gamma}) \}
= O(1) \times \exp\{- c_3 (\log T ) (d_{\mo_j}-d_{\mo_i})\}
.
$$
for $d_{\mo_i} < d_{\mo_j} \leq c_4(T/\log T)^{1/(1+\gamma)}$ for some constants $c_3,c_4$.
This indicates that the relative weights of underfitting models will exponentially decay {\co as the dimension departs from the optimum}. We leave a more sophisticated analysis for future research. 

\begin{remark}[More on Algorithm~\ref{algo:trackExpert}]
{\co We note that the penalization method in the sequential algorithm can be replaced with a general black-box method that computes $\Lc[\mo]$, a quantity that approximates the out-sample prediction loss. Proposition 1 is not necessarily specific to GTIC. }
Also, an anonymous reviewer pointed out that the above sequential model selection would be most useful when combined with a sequential update of the model parameters (for a given model). Suppose that $\hat{\theta}_n$ is the MLE of a parameter $\theta$ under $t$ data observations. Here we summarize two common methods that could be potentially used in online implementations: 1) calculate asymptotic expression of $\hat{\theta}_t - \hat{\theta}_{t-1}$, also called the asymptotic influence function, to update from $\hat{\theta}_{t-1}$ to $\hat{\theta}_{t}$ (see e.g. \cite[Theorem 5.23]{van2000asymptotic}), and 2) use $\hat{\theta}_{t-1}$ as a warm start for estimating $\hat{\theta}_{t}$ when using iterative algorithms such as the (stochastic) gradient descent and Newton-Raphson method.
\end{remark}

%
%
\section{Conclusion} 

In the framework of parametric models with possibly expanding model dimensions and model space, we studied a method to approach the limit of statistical learning in the sense that the predictive power of the selected model is asymptotically close to the best offered from a model class. 
The proposed method, GTIC, is an extension of an information criterion by Takeuchi to more general loss functions.  
Our theoretical analysis of GTIC justifies the use of TIC for general mis-specified model classes, and extends some technical tools for classical analysis of AIC in linear models. Moreover, the proposed approach serves as an alternative of leave-one-out cross-validation that is in general not accessible due to its computational burden.
In the second part of the paper, we also proposed a sequential model expansion algorithm for reliable online prediction with low computation cost, based on our new graph-based expert tracking techniques.  
In summary, the proposed methodology is asymptotically optimal and practically useful, and it can be a promising competitor of cross-validation in both batch and online settings.   

%

%
%
%
%


\appendices

%
%

\section{Proof of Theorem~\ref{theorem:efficiency}}
{\co 
We first outline the proof of Theorem~\ref{theorem:efficiency}.
Lemma~\ref{lemma:consistency} proves that $\thE[\mo]$ is $n^\tau$-consistent.
Lemma~\ref{lemma:vol} bounds the volume of a neighborhood of the oracle model using the bracketing number. Lemma~\ref{lemma:entropyIneq} is a technical result that relates the volume of a union of model spaces with that of individual ones.
Based on the above lemmas, we show sufficient conditions to guarantee $\sup_{\bm f \in \F_n} \norm{\G \bm f} =o_p(1)$ in Lemma~\ref{lemma:tight}, and its specialization when $\bm f $ is in the form of $ \bm \psi_{n}(\cdot, \thE[\mo] ; \mo) - \bm \psi_{n}(\cdot, \thT[\mo] ; \mo)$ in Lemma~\ref{lemma:locUnifConv} (which states a type of local uniform convergence). 
We then prove Lemma~\ref{lemma:normality}, a counterpart of the classical asymptotic normality of maximum likelihood estimators. All the lemmas are assembled in the final proof of Theorem~\ref{theorem:efficiency} with Taylor expansions.
 
}





\begin{lemma} \label{lemma:consistency}
	Suppose that Assumptions~\ref{ass:data},~\ref{ass:unique},~\ref{ass:converge},~\ref{ass:eig},~\ref{ass:Lip},~\ref{ass:secondOrderLLN} hold. Then $\thE$ is $n^{\tau}$-consistent uniformly over $\Mo$, namely $\sup_{\mo \in \Mo} n^{\tau} \norm{\thE[\mo]-\thT[\mo]} = O_p(1)$.
\end{lemma}

\begin{proof}
	 Using Assumptions~\ref{ass:data}, \ref{ass:unique}, and a direct adaptation of the techniques in \cite[Theorem 5.7]{van2000asymptotic} {\co (which is on the asymptotic consistency of M-estimators)}, we can prove that $\thE[\mo]$ is consistent in the sense that 
	 \begin{align}
	 	\sup_{\mo \in \Mo} 	\norm{\thE[\mo] - \thT[\mo]} = o_p(1) \label{eq70}
 	 \end{align}
	 as $n \rightarrow \infty$.
	 
	 From the definitions of $\thE$ and $\thT$, we have for each $\mo \in \Mo$
	 \begin{align}
	 	&n^{\tau} \E \{\bm \psi_{n}(\cdot, \thT[\mo] ; \mo)-\bm \psi_{n}(\cdot, \thE[\mo] ; \mo)\} \nonumber\\
	 	&= n^{\tau} \{0 - \E \bm \psi_{n}(\cdot, \thE[\mo] ; \mo)\} \nonumber\\
	 	&= n^{\tau} \{E_n \bm \psi_{n}(\cdot, \thE[\mo] ; \mo) - \E \bm \psi_{n}(\cdot, \thE[\mo] ; \mo)\}	\label{eq31}
	 \end{align}
	 
	From the differentiability of the map $\th \mapsto \E \psi_{n}(\cdot,\th ; \mo)$, there exists $\tilde{\th}[\mo]$ such that $\norm{\tilde{\th}[\mo] - \thT[\mo]} \leq \norm{\thE[\mo] - \thT[\mo]}$, and 
	\begin{align} 
		&\E \{\bm \psi_{n}(\cdot,\thT[\mo] ; \mo)-\bm \psi_{n}(\cdot,\thE[\mo] ; \mo)\} \nonumber\\ 
		 &= \nabla_{\th} \E \{\bm \psi_{n}(\cdot,\tilde{\th}[\mo] ; \mo) \} (\thT[\mo]-\thE[\mo]) \nonumber \\
		 &= V_{n}(\tilde{\th}[\mo] ;\mo) (\thT[\mo]-\thE[\mo]) , \label{eq32}
	\end{align}
	where the exchangeability of integral and differentiation (in the second identity) is guaranteed by (\ref{eq:Lip}) and (\ref{eq:exchangeable}) in Assumption~\ref{ass:Lip}. 
	
	Therefore, with probability tending to one, we have 
	\begin{align}
		&\sup_{\mo \in \Mo} n^{\tau} \norm{ V_{n}(\tilde{\th}[\mo] ;\mo) (\thT[\mo]-\thE[\mo]) } \nonumber\\
		&= \sup_{\mo \in \Mo} n^{\tau} \norm{ \E \{\bm \psi_{n}(\cdot,\th ; \mo)-\bm \psi_{n}(\cdot,\thE ; \mo)\} } \nonumber\\
		&= \sup_{\mo \in \Mo} n^{\tau} \norm{ E_n \bm \psi_{n}(\cdot, \thE[\mo] ; \mo) - \E \bm \psi_{n}(\cdot, \thE[\mo] ; \mo) } \nonumber\\
		&= O_p(1) \nonumber	
	\end{align}
	where the first equality is due to (\ref{eq32}), the second equality is due to (\ref{eq31}), and the third equality comes from Assumption~\ref{ass:converge}. 
	By the (\ref{eq:uniformV}) in Assumption~\ref{ass:secondOrderLLN} and Assumption~\ref{ass:eig}, $V_{n}(\tilde{\th}[\mo]; \mo)$ is invertible for each $\mo \in \Mo$, and  
	$$\sup_{\mo \in \Mo} \norm{V_{n}(\tilde{\th}[\mo]; \mo)^{-1} } < 1/(2c_1)$$ 
	with probability tending to one. It follows that  
	\begin{align}
		&\sup_{\mo \in \Mo} n^{\tau} \norm{\thT[\mo]-\thE[\mo]} \nonumber\\
		&\leq \sup_{\mo \in \Mo} \biggl\{ 
		\norm{V_{n}(\tilde{\th}[\mo]; \mo)^{-1}} \cdot
		\norm{ n^{\tau} V_{n}(\tilde{\th}[\mo]; \mo) (\thT[\mo]-\thE[\mo])} 	\biggr\} \nonumber \\
		&= O_p(1) ,
	\end{align}
	which concludes the proof. 
\end{proof}


Before we proceed, we need the following definition. 

\begin{definition}[Bracketing number] \label{def:bracketing}
	Given two scalar functions $f_1$ and $f_2$, the bracket $[f_1,f_2]$ is the set of all functions $f$ such that $f_1 \leq f \leq f_2$. An $\v$-bracket in $L_2(\P)$ is a bracket $[f_1,f_2]$ with $\E (f_2-f_1)^2 < \v^2$. The bracketing number $N_{[\ ]}(\v,\F,L_2(\P))$ is the minimum number of $\v$-brackets needed to cover a set $\F$.  Moreover, the bracketing integral is defined by 
	\begin{align}
	I_{[\ ]}(\delta, \F, L_2(\P)) = \int_{0}^{\delta} \sqrt{\log N_{[\ ]}(\v,\F,L_2(\P))} d\v 
	\label{eq71}
	\end{align}
	for $\delta > 0$.
\end{definition}

	The logarithm of the above bracketing number is also referred to as bracketing entropy relative to the $L_2(\P)$-norm. It is commonly used to describe the size of a class of functions. 
	We will use the above definition in order to prove uniform convergence results.  
	We refer to \cite{spokoiny2012parametric} for a different bracketing idea used to study the nonasymptotic estimation theory. 

We have the following lemma whose proof follows directly from Definition~\ref{def:bracketing} and Assumption~\ref{ass:Lip}. 
		
\begin{lemma} \label{lemma:vol}
	Suppose that Assumption~\ref{ass:Lip} holds, and $r_n \leq r$ for all $n$ (where $r$ has been defined in Assumption~\ref{ass:Lip}). Let $\F_n[\mo]=\bigl\{
			\bm \psi_{n}(\cdot, \th ; \mo): \th \in B(\thT[\mo],r_n) \subset \Real^{\d} 
		\bigr\}$ be a collection of (vector-valued) measurable functions. 
	Then 
	\begin{align}
		N_{[\ ]}(\v , \F_n[\mo], L_2(\P)) \leq \bigl(\v^{-1} r_n \norm{m_n}_{\P}\bigr)^{\d}  \nonumber	
	\end{align}
	 for all $0 < \v < r_n$. 
\end{lemma}


We prove the following technical lemmas. 

\begin{lemma} \label{lemma:entropyIneq}
	For any sets of functions $\F_j,j=1,\ldots,k$, we have 	
	\begin{align}
		&I_{[\ ]}(\delta, \cup_{1 \leq j \leq k}\F_j , L_2(\P)) 
		\leq 
			\delta \sqrt{ \log k} + 
			k \max_{1 \leq j \leq k} I_{[\ ]}(\delta, \F_j, L_2(\P)) 
		 \nonumber	
	\end{align}
\end{lemma}

\begin{proof}
{\co
	Direct calculations show that 
	\begin{align*}
	\int_0^\delta \sqrt{\log N_{[\ ]}(\v, \cup_{j=1}^k \mathcal{F}_j, L_2(P_*))} \mathrm{d}\v
	& \leq \int_0^\delta \sqrt{\log \sum_{j=1}^k N_{[\ ]}(\v , \mathcal{F}_j, L_2(P_*))} \mathrm{d}\v \\
	& \leq \delta \sqrt{\log k} + \int_0^\delta \max_{1 \leq j \leq k} \sqrt{\log N_{[\ ]}(\v , \mathcal{F}_j, L_2(P_*))} \mathrm{d}\v \\
	& \leq \delta \sqrt{\log k} + \int_0^\delta \sum_{j=1}^k \sqrt{\log N_{[\ ]}(\v , \mathcal{F}_j, L_2(P_*))} \mathrm{d}\v  \\
	& \leq \delta \sqrt{\log k} + \sum_{j=1}^k \int_0^\delta  \sqrt{\log N_{[\ ]}(\v , \mathcal{F}_j, L_2(P_*))} \mathrm{d}\v \\
	& \leq \delta \sqrt{\log k} + k \max_{1 \leq j \leq k} I_{[\ ]}(\delta, \mathcal{F}_j, L_2(P_*)) .
\end{align*}
}
\end{proof}

%

\begin{definition}
For any class $\F$ of functions $f: \Z \rightarrow \Real$, a function $F: \Z \rightarrow \Real$ is called an envelope function of $\F$, if $\sup_{f \in \F} |f(z)| \leq F(z) < \infty $ for every $ z \in \Z$.	
\end{definition}

\begin{lemma}\cite[Lemma 19.34]{van2000asymptotic} \label{lemma:ineq}
	For any class $\F$ of measurable functions $f: \D \rightarrow \Real$ such that $\E f^2 < \delta^2$ for all $f$, {\co we have}, with 
	$$a(\delta) = \delta / \sqrt{\max\{1, \log N_{[\ ]}(\delta,\F,L_2(\P)) \}}$$ 
	and $F$ an envelope function, that 
	\begin{align}
		\E \sup_{f\in \F}| \G f | \lesssim I_{[\ ]}(\delta,\F,L_2(\P)) + \sqrt{n}\E \bigl\{ F \cdot 1_{F > \sqrt{n} a(\delta)} \bigr\}\nonumber	.
	\end{align}
	Here, $1_{A}$ is the indicator function of event $A$. 
\end{lemma}

\begin{lemma} \label{lemma:tight}
	Let $\F_n=\cup_{\mo \in \Mo}\F_n[\mo]$, where $\F_n[\mo] = \{\bm f_{n,u}:u \in U[\mo] \}$ is a class of measurable vector-valued functions. In other words, for each $\mo \in \Mo$ and $u \in U[\mo]$, $\bm f_{n,u}=[f_{n,u,1}, \ldots, f_{n,u,\d}]^\T$ with $f_{n,u,i} : \D \rightarrow \Real$ being a scalar-valued function. The dimension $d_n[\alpha]$ may depend on $\alpha$ and $n$, and we let $d_n = \max_{\mo \in \Mo} \d$. Assume that the following conditions hold.

	(i)
	There is an envelope function $F_{n}$ that satisfies 
	$$\sup_{\mo \in \Mo, u \in U[\mo] \subset \Real^{\d}, 1 \leq i \leq \d} | f_{n,u,i}(\bm z)| \leq F_{n}(\bm z) < \infty $$ 
	for every $\bm z \in \D$;
	
	(ii) There exists a deterministic sequence $\{\delta_n\}$ such that
	\begin{align}
		\p\sqrt{\log \{ \p \card(\Mo) \}} \delta_n \rightarrow 0 ;	\label{eq75}
	\end{align}
	

	
	(iii) The bounded moment condition:
	\begin{align}
		&\delta_n^{-2} \E F_{n}^2 \rightarrow 0 ;\nonumber	
	\end{align}

	(iv) The bounded class condition: 
	$$\p^{2} \ \card(\Mo) \times \sup_{\mo \in \Mo, 1 \leq i \leq \d} I_{[\ ]}(\delta_n,\F_{n,i}[\mo],L_2(\P)) \rightarrow 0, $$
	where we let $\F_{n,i}[\mo] = \{f_{n,u,i} : u \in U[\mo] \}$. 
	
	Then we have $$\sup_{\bm f \in \F_n} \norm{\G \bm f} \limp 0$$ as $n \rightarrow \infty$. 
\end{lemma}

\begin{proof}
	By Markov's inequality, it suffices to prove that $\E \sup_{\mo \in \Mo, u\in U[\mo]} \norm{\G \bm f_{n,u}} \rightarrow 0$ as $n\rightarrow 0$. 
	
	Condition (iii) implies that for all sufficiently large $n$, 
	\begin{align}
	\sup_{\mo \in \Mo, u \in U[\mo], i=1,\ldots,\p} \E f_{n,u,i}^2  
	&\leq 
	\E \sup_{\mo \in \Mo, u \in U[\mo], i=1,\ldots,\p} f_{n,u,i}^2 \nonumber \\
	&< \delta_n^2. \label{eq74}
	\end{align}
	Let $\delta_n,a_{n}(\delta_n)$ be the constants given in Lemma~\ref{lemma:ineq} corresponding to $ \delta=\delta_n$ and 
	$$\tilde{\F}_n = \bigcup_{\mo \in \Mo, 1 \leq i \leq \d}\F_{n,i}[\mo].$$
	From inequality (\ref{eq74}) and Lemma~\ref{lemma:ineq}, we have 
	\begin{align}
		&\E \sup_{\mo \in \Mo, u\in U[\mo], 1 \leq i \leq \d} |\G f_{n,u,i}| \nonumber\\
		&\lesssim I_{[\ ]}(\delta_n, \tilde{\F}_n, L_2(\P)) + \sqrt{n}\E 
		\bigl\{ F_{n} \cdot 1_{F_{n} > \sqrt{n} a_n(\delta_n) } \bigr\} \nonumber \\
		&\leq I_{[\ ]}(\delta_n,\tilde{\F}_n,L_2(\P)) + \frac{1}{a_{n}(\delta_n)} \E F_{n}^2  ,\label{eq73} 
	\end{align}
	where the second inequality comes from the fact that 
	$$
	1_{F_{n} > \sqrt{n} a_n(\delta_n) } \leq \frac{F_{n}}{\sqrt{n} a_n(\delta_n)} 1_{F_{n} > \sqrt{n} a_n(\delta_n) } \leq \frac{F_{n}}{\sqrt{n} a_n(\delta_n)} .
	$$
	By the definition of $a_n(\cdot)$, $I_{[\ ]}(\delta,\tilde{\F}_n,L_2(\P))$, and the fact that $N_{[\ ]}(\delta,\tilde{\F}_n,L_2(\P))$ is non-increasing in $\delta$, 
	we have 
	\begin{align*}
	 	\frac{1}{a_{n}(\delta_n)}
	 	&=
		\frac{1}{\delta_n} \sqrt{\max\{1, \log N_{[\ ]}(\delta_n,\tilde{\F}_n,L_2(\P)) \}}\\
		&\leq 	
		\frac{1}{\delta_n^2} I_{[\ ]}(\delta_n,\tilde{\F}_n,L_2(\P)).
	\end{align*}
	It follows that the right hand side of (\ref{eq73}) is upper bounded by 
	$$
	I_{[\ ]}(\delta_n,\tilde{\F}_n,L_2(\P)) \bigl(1 + \delta_n^{-2} \E F_{n}^2 \bigr) . 
	$$

	Therefore, by Lemma~\ref{lemma:entropyIneq} and simple manipulations, we have 
	\begin{align}
		&\E \sup_{\mo \in \Mo, u\in U[\mo]} \norm{\G \bm f_{n,u}} \nonumber\\
		&\leq \E \sup_{\mo \in \Mo, u\in U[\mo]} \sum_{i=1}^{\p} |\G f_{n,u,i}| \nonumber\\
		&\leq \p \E \sup_{\mo \in \Mo, u\in U[\mo], 1 \leq i \leq \d} |\G f_{n,u,i}|	\nonumber\\
		&\leq (A_1 + A_2 ) \bigl(1 + \delta_n^{-2 } \ \E F_{n}^2 \bigr) ,
		 \label{eq45}	
	\end{align}
	where 
	\begin{align*}
		A_1 &= \p\sqrt{ \log \{ \p \card(\Mo) \}} \delta_n , \\
		A_2 &= {\co \p^{2} \card(\Mo) \sup_{\mo \in \Mo, 1 \leq i \leq \d} I_{[\ ]}(\delta_n, \F_{n,i}[\mo], L_2(\P)) .}
	\end{align*}
	Assumptions~(ii), (iii), and (iv) guarantee that the right hand side of (\ref{eq45}) goes to zero as $n \rightarrow \infty$, which concludes the proof. 
	
\end{proof}

Using the above results, we can prove the following key technical lemma. 

\begin{lemma} \label{lemma:locUnifConv}
	Suppose that Assumptions~\ref{ass:data}-
	\ref{ass:secondOrderLLN} hold. 
	Then 
	\begin{align}
		\sup_{\mo \in \Mo} \norm{
		\G \bm \psi_{n}(\cdot, \thE[\mo] ; \mo) - \G \bm \psi_{n}(\cdot, \thT[\mo] ; \mo) }
		= o_p(1) . \label{eq:key}
	\end{align}
 	
\end{lemma}

\begin{proof}
%
	For a constant $c$, consider the class $\F_n=\cup_{\mo \in \Mo}\F_n[\mo]$, with $\F_n[\mo] \de \{\bm f_{n, \bm u}: \bm u \in U[\mo] \}$, $U[\mo] = \{[u_1,\ldots,u_{\d}]^T: \sum_{i=1}^{\d} u_i^2=c \}$, and 
	$$
	\bm f_{n,\bm u}(\cdot) = 
	\bm \psi_{n}(\cdot, \thT[\mo]+n^{-\tau} \bm u; \mo)
	-
	\bm \psi_{n}(\cdot, \thT[\mo] ; \mo) .
	$$
	
	 Suppose that $\v,\delta>0$ are fixed constants. It suffices to prove that the left hand side of (\ref{eq:key}) is less than $\delta$ with probability at least $1-\v$ for all sufficiently large $n$. By Lemma~\ref{lemma:consistency}, there exists a constant $c>0$ such that $\G \bm \psi_{n}(\cdot, \thE[\mo] ; \mo) - \G \bm \psi_{n}(\cdot, \thT[\mo] ; \mo)$ falls into the class $\F_n$ with probability at least $1-\v/2$ for all sufficiently large $n$. Therefore, we only need to prove that for any given constant $c>0$, $\sup_{\bm f \in \F_n} \norm{\G \bm f} \limp 0$. 
	 It remains to prove that there are $\delta_n$'s that satisfy Conditions (i)-(iv) of Lemma~\ref{lemma:tight}. 
	 
 	 We define $m_n (\cdot) = \sup_{\mo \in \Mo} m_n[\mo](\cdot)$. 
	 By Assumption~\ref{ass:Lip}, we can use $F_n(\cdot) \de c n^{-\tau} \sup_{\mo \in \Mo} m_n[\mo](\cdot) $ as the envelop function for each $f_{n,u,i}(\cdot)$, 
	 and we have 
	 \begin{align}
	 	\E F_{n}^2
	 	 &\leq C_1 \de c^2 n^{-2\tau} \E m_n^2 \nonumber . 
	 \end{align}
	 Let 
	 $$
	 	C_2 = \p\sqrt{\log \{ \p \card(\Mo) \}}. 
	 $$
	 Because of (\ref{eq:mCond}) in Assumption~\ref{ass:Lip}, we have  
	 \begin{align}
	 	C_2^2 C_1 
	 	&= c^2 n^{-2\tau} \p^2 \log \{ \p \card(\Mo) \} \ \E m_n^2 \rightarrow 0 .
	 \end{align}
	 This implies the existence of a sequence $\delta_n$ (e.g. $\delta_n = C_1^{1/4} C_2^{-1/2}$) such that 
	 \begin{align*}
	 	\delta_n C_2 \rightarrow 0, \quad 	\delta_n^{-2} C_1 \rightarrow 0 ,
	 \end{align*}
	 which further implies Conditions~(ii) and (iii) in Lemma~\ref{lemma:tight}. 
	  	
 	To conclude the proof, we prove that Condition (iv) in Lemma~\ref{lemma:tight} holds for any $\delta_n \rightarrow 0$. 
 	From Lemma~\ref{lemma:vol}, we have for each $\mo \in \Mo, 1 \leq i \leq \d$ that 
 	\begin{align}
 		&I_{[\ ]}(\delta_n,\F_{n,i}[\mo],L_2(\P)) \nonumber\\
 		&\leq
 		 \int_0^{\delta_n} \biggl[\max\biggl\{0, \p \log \bigl(\v^{-1} c n^{-\tau} \norm{m_n}_{\P} \bigr) \biggr\}\biggr]^{1/2}	d \v \nonumber\\
 		&= \int_0^{\min\{\delta_n, c n^{-\tau} \norm{m_n}_{\P}\}} \biggl[ \p \log \bigl(\v^{-1} c n^{-\tau} \norm{m_n}_{\P} \bigr) \biggr]^{1/2}	d \v . \label{eq47}
 	\end{align}
 	Because condition (\ref{eq:mCond}) implies that $n^{-\tau} \norm{m_n}_{\P} \rightarrow 0$, the value of $\v$ in the integral is close to zero. This implies that for all sufficiently large $n$, the integrand in (\ref{eq47}) is upper bounded by 
 	$
 		\p^{1/2} \v^{-\rho}, 
 	$
 	where $\rho$ is chosen such that $1/(1-\rho) = \gamma$, and $\gamma$ is given in Assumption~\ref{ass:Lip}.
 	Therefore, for all sufficiently large $n$, the right-hand side of (\ref{eq47}) is upper bounded by 
 	\begin{align}
 		 \int_0^{c n^{-\tau}\norm{m_n}_{\P} } \p^{1/2} \v^{-\rho} d\v 
 		= (1-\rho)^{-1} \p^{1/2} (c n^{-\tau} \norm{m_n}_{\P})^{1-\rho} ,
 		\nonumber 
 	\end{align}
 	which does not depend on $\mo,i$. This further implies 
 	\begin{align}
 		&\p^2 \ \card(\Mo) \times \sup_{\mo \in \Mo, 1 \leq i \leq \d} I_{[\ ]}(\delta_n,\F_{n,i}[\mo],L_2(\P)) \nonumber \\
 		&\leq 
 		(1-\rho)^{-1} c^{1-\rho} \p^2 \card(\Mo) (n^{-\tau} \norm{m_n}_{\P})^{1-\rho} \nonumber \\
 		&=(1-\rho)^{-1} c^{1-\rho} \biggl( \p^{2\gamma} \ \card(\Mo)^{\gamma} \ n^{-\tau} \norm{m_n}_{\P} \biggr)^{1-\rho} \nonumber\\
 		&\rightarrow 0 ,
 	\end{align}
 	where the last limit is due to (\ref{eq:mCond}) in Assumption~\ref{ass:Lip}. 
\end{proof}


Next, we prove the second key technical lemma. 

\begin{lemma} \label{lemma:normality}
	Suppose that Assumptions~\ref{ass:data}-
	\ref{ass:secondOrderLLN} hold.
	Assume that the map $\th \mapsto \E \bm \psi_{n}(\cdot, \th ; \mo)$ is differentiable at a $\thT$ for all $n$. 
	Then we have 
	\begin{align*}
		\sqrt{n} (\thE[\mo]-\thT[\mo]) 
		= & - \{ V_n(\thT[\mo] ; \mo)^{-1} + \nu_{1,n}[\mo] \} \times 
		\frac{1}{\sqrt{n}} \sum_{i=1}^n \bm \psi_{n}(\bm z_i, \bm \thT[\mo] ; \mo) + \nu_{2,n}[\mo] , 
	\end{align*}
	where $\nu_{1,n}[\mo]$ is a positive semidefinite matrix and $\bm \nu_{2,n}[\mo]$ is a vector such that $\sup_{\mo \in \Mo} \norm{\nu_{1,n}[\mo]} \limp 0$ and $\sup_{\mo \in \Mo} \norm{\bm \nu_{2,n}[\mo]}\limp 0$. 
\end{lemma}

\begin{proof}
	By the definitions of $\thT$ and $\thE$, we have 
	\begin{align}
		&\sqrt{n} \E \{ \bm \psi_n (\cdot, \thE[\mo] ; \mo)-\bm \psi_n(\cdot, \thT[\mo] ; \mo)\}	\nonumber\\
		&= \sqrt{n} \{\E \bm \psi_n (\cdot, \thE[\mo] ; \mo)-0\}	\nonumber\\
		&= \sqrt{n} \{ \E \bm \psi_n (\cdot, \thE[\mo] ; \mo)-E_n \bm \psi_n (\cdot, \thE[\mo] ; \mo)\} \nonumber \\
		&= -\G \bm \psi_n (\cdot, \thE[\mo] ; \mo) \nonumber \\
		&=-\G \bm \psi_n (\cdot, \thT[\mo] ; \mo) +\bm \nu_n \label{eq321}
	\end{align}
	where the last equality is due to Lemma~\ref{lemma:locUnifConv}, and $\norm{\bm \nu_n}=o_p(1)$. 
	
	From the differentiability of the map $\th \mapsto \E \psi_{n}(\cdot,\th ; \mo)$, there exists $\tilde{\th}[\mo]$ such that $\norm{\tilde{\th}[\mo] - \thT[\mo]} \leq \norm{\thE[\mo] - \thT[\mo]}$, and 
	\begin{align} 
		&\E \{\bm \psi_{n}(\cdot,\thT[\mo] ; \mo)-\bm \psi_{n}(\cdot,\thE[\mo] ; \mo)\} \nonumber\\ 
		 &= \nabla_{\th} \E \{\bm \psi_{n}(\cdot,\tilde{\th}[\mo] ; \mo) \} (\thT[\mo]-\thE[\mo]) \nonumber \\
		 &= V_{n}(\tilde{\th}[\mo] ;\mo) (\thT[\mo]-\thE[\mo]) , \label{eq320}
	\end{align}
	where the exchangeability of integral and differentiation (in the second identity) is guaranteed by (\ref{eq:Lip}) and (\ref{eq:exchangeable}) in Assumption~\ref{ass:Lip}. 
	Multiplying the matrix $\sqrt{n} V_{n}(\tilde{\th}[\mo] ;\mo)^{-1}$ to both sides of (\ref{eq320}) and using equality (\ref{eq321}), we have 
	\begin{align}
		\sqrt{n} (\thE[\mo] - \thT[\mo]) 
		= & - V_{n}(\tilde{\th}[\mo] ;\mo)^{-1} \G \bm \psi_{n}(\cdot,\thT[\mo] ; \mo) + 
		V_{n}(\tilde{\th}[\mo] ;\mo)^{-1} \bm \nu_n \label{eq48}.
	\end{align}
	We conclude the proof by applying Assumption~\ref{ass:eig} (with the constant $c_2$) and (\ref{eq:consistV}) in Assumption~\ref{ass:secondOrderLLN} to equality (\ref{eq48}).

 
%
\end{proof}

\textbf{Proof of Theorem~\ref{theorem:efficiency}}
 
 	In order to prove that the minimum of $\Lcn[\mo]$ asymptotically approaches the minimum of $\L[\mo]$ (in the sense of Definition~\ref{def:efficiency}), we only need to prove that $\Lcn[\mo] / \L[\mo]=1+o_p(1) $ where $o_p(1)$ is uniform in $\mo \in \Mo$. In other words, 
 	\begin{align}
 		\sup_{\mo \in \Mo} \biggl| \frac{\Lcn[\mo]-\L[\mo]}{\L[\mo]} \biggr| \limp 0. \nonumber
 	\end{align}
	Recall the definition of $\R[\mo]$. It further suffices to prove that 
	\begin{align}
		\sup_{\mo \in \Mo} \biggl| \frac{\Lcn[\mo]-\L[\mo]}{\R[\mo]} \biggr| \limp 0, \label{eq:o}	
	\end{align}
	and
	\begin{align}
		\sup_{\mo \in \Mo} \frac{\L[\mo]}{\R[\mo]} 	 \limp 1 . \label{eq83}
	\end{align}

	By the definition of loss $\L[\mo]$ and Taylor expansion, we have for each $\mo \in \Mo$ 
	\begin{align}
		\L[\mo]
		&= \E \l(\bm z, \thE[\mo]; \mo) \nonumber \\
		&= \E \l(\bm z , \thT[\mo] ; \mo) + (\thE[\mo]-\thT[\mo])^\T 
		\frac{\partial}{\partial \bm \th} \E \l(\bm z , \thT[\mo] ; \mo) + 
		 \frac{1}{2} \norm{\thE[\mo]-\thT[\mo]}_{ \nabla_{\bm \th}^2 \E \l(\bm z , \tilde{\bm \theta}[\mo]; \mo) }^2  
		\nonumber \\
		&= \E \l(\bm z , \thT[\mo] ; \mo) + \frac{1}{2} \norm{\thE[\mo]-\thT[\mo]}_{ V_n(\tilde{\bm \theta}[\mo]; \mo) }^2 \label{eq80}
	\end{align}
	where $\tilde{\bm \theta}[\mo]$ in the second equality is a vector satisfying $\norm{\tilde{\bm \theta}[\mo]-\thT[\mo]} \leq \norm{\thE[\mo]-\thT[\mo]}$, and the exchangeability of expectation and differentiation in the third equality is guaranteed by (\ref{eq:exchangeable}) in Assumption~\ref{ass:Lip}, and the consistency of $\thE[\mo]$.
	We note that by Assumption~\ref{ass:eig}, the equality (\ref{eq80}) further implies (\ref{eq80_new}) presented in our introduction. 
	
	Similarly, we have 
	\begin{align}
		\Lhat[\mo] 
		&= \frac{1}{n} \sum_{i=1}^n \l(\bm z_i , \thE[\mo] ; \mo) \nonumber\\
		&= \frac{1}{n} \sum_{i=1}^n \l(\bm z_i , \thT[\mo] ; \mo) + 
		(\thE[\mo]-\thT[\mo])^\T \frac{1}{n} \sum_{i=1}^n \bm \psi_n(\bm z_i,\thT[\mo] ; \mo) + 
		\frac{1}{2} \biggl\lVert \thE[\mo]-\thT[\mo] \biggr\rVert_{\hat{V}_n(\doublewidetilde{\bm \theta}[\mo])}^2 \label{eq81}.
	\end{align}

	From identities (\ref{eq80}) and (\ref{eq81}), we may write 
	\begin{align*}
		&\L[\mo] - \Lhat[\mo] - 	\frac{1}{n} \tr\biggl\{ \hat{V}_n(\thE[\mo]; \mo)^{-1} \hat{J}_n(\thE[\mo]; \mo)\biggr\} \\
		&= A_3[\mo] + A_4[\mo] + A_5[\mo] + A_6[\mo] 
	\end{align*}
	where we define
	\begin{align*}
		A_3[\mo] &= \frac{1}{2} \norm{\thE[\mo]-\thT[\mo]}_{ V_n(\tilde{\bm \theta}[\mo]; \mo) - \hat{V}_n(\doublewidetilde{\bm \theta}[\mo]) }^2 \\
		A_4[\mo] &= -\frac{1}{n} \sum_{i=1}^n \{\l(\bm z_i, \thT ; \mo)- \E \l(\bm z, \thT[\mo]; \mo) \} \\
		A_5[\mo] &= \frac{1}{n} \biggl\{ \tr \bigl\{ V_n(\thT[\mo]; \mo)^{-1} J_n(\thT[\mo]; \mo) \bigr\}  - \tr\bigl\{ \hat{V}_n(\thE[\mo]; \mo)^{-1} \hat{J}_n(\thE[\mo]; \mo)\bigr\} \biggr\} \\
		A_6[\mo] &= -(\thE[\mo]-\thT[\mo])^{\T} \frac{1}{n} \sum_{i=1}^n \bm \psi_n(\bm z_i,\thT[\mo] ; \mo) - \frac{1}{n} \tr \bigl\{ V_n(\thT[\mo]; \mo)^{-1} J_n(\thT[\mo]; \mo) \bigr\}. 
	\end{align*}
	
	In view of (\ref{eq:o}), it suffices to prove that 
	\begin{align}
		\sup_{\mo \in \Mo} \frac{|A_k[\mo]|}{\R[\mo]} \limp 0 \label{eq82}
	\end{align} as $n \rightarrow \infty$ for $k=3,4,5,6$, and the limit (\ref{eq83}).

	By the $n^\tau$-consistency of $\thE[\mo]$ uniformly over $\Mo$ (Lemma~\ref{lemma:consistency}) and Assumption~\ref{ass:secondOrderLLN}, 
	\begin{align*}
		 \sup_{\mo \in \Mo} \frac{|A_3[\mo]|}{\R[\mo]}  
		&= \sup_{\mo \in \Mo}	\frac{1}{2}\frac{n^{-2\tau}}{\R[\mo]} \norm{\bm \nu_n}_{ V_n(\tilde{\bm \theta}[\mo]; \mo) - \hat{V}_n(\doublewidetilde{\bm \theta}[\mo]) }^2 
	\end{align*}
	where $\sup_{\mo \in \Mo} \norm{\bm \nu_n} = O_p(1)$. 
	Thus, given assumption (\ref{eq84}), (\ref{eq82}) with $k=3$ can be proved. 
	
	By Chebyshev's inequality, for any positive constant $\delta>0$, we have 
	\begin{align}
		&\P \biggl( \sup_{\mo \in \Mo} \frac{|A_4[\mo]|}{\R[\mo]} > \delta \biggr) \nonumber\\ 
		&\leq \sum_{\mo \in \Mo} \P \biggl( \frac{|A_4[\mo]|}{\R[\mo]} > \delta \biggr) 
		\nonumber \\
		&\leq 	\sum_{\mo \in \Mo} \frac{\E \{\l(\bm z_1, \thT ; \mo)- \E \l(\bm z, \thT[\mo]; \mo) \}^{2m_1}}{ \delta^{2m_1} n^{2m_1} \R[\mo]^{2m_1} }.\label{eq90}
	\end{align}
	Thus, given assumption (\ref{eq85}), (\ref{eq82}) with $k=4$ can be proved. 

	For brevity, we temporarily denote 
	$$V_n(\thT[\mo]; \mo), \ \hat{V}_n(\thE[\mo]; \mo), \ J_n(\thT[\mo]; \mo), \textrm{ and } \hat{J}_n(\thE[\mo]; \mo)$$ respectively by 
	$$V[\mo], \ \hat{V}[\mo], \ J[\mo],\textrm{ and } \hat{J}[\mo]. $$ 
	Then 
	\begin{align*}
		&\tr\{ V[\mo]^{-1} J[\mo] \} - \tr\{ \hat{V}[\mo]^{-1} \hat{J}[\mo] \} \\
		&= \tr\{ V[\mo]^{-1} ( J[\mo] - \hat{J}[\mo] ) \} + \tr\{ (V[\mo]^{-1} - \hat{V}[\mo]^{-1}) \hat{J}[\mo] \} .
	\end{align*}
	To prove (\ref{eq82}) with $k=5$, we only need to show that 
	\begin{align}
	&\sup_{\mo \in \Mo} \frac{1}{n \R[\mo]} \tr\{ V[\mo]^{-1} ( J[\mo] - \hat{J}[\mo] ) \} \limp 0 , \label{eq86} \\
	&\sup_{\mo \in \Mo} \frac{1}{n \R[\mo]} \tr\{ (V[\mo]^{-1} - \hat{V}[\mo]^{-1}) \hat{J}[\mo] \} \limp 0 . \label{eq87}
	\end{align}  
	We only prove (\ref{eq86}), and then (\ref{eq87}) follows similar arguments. 
	Suppose that $\bm z$ is a $\mathcal{N}(0,I)$ random variable of dimension $\d$, and $V[\mo]^{-1/2}$ is a positive semidefinite matrix whose square equals $V[\mo]^{-1}$.  
	Because of Assumption~\ref{ass:eig} and \ref{ass:Lip}, (\ref{eq86}) could be rewritten as 
	\begin{align*}
		&\sup_{\mo \in \Mo} \frac{1}{n \R[\mo]} E\biggl\{ \bm z^\T V[\mo]^{-1/2} ( J[\mo] - \hat{J}[\mo] ) V[\mo]^{-1/2} \bm z \biggr\}  \\
		&= o_p(1) \sup_{\mo \in \Mo} \frac{1}{n \R[\mo]} E\norm{ V[\mo]^{-1/2} \bm z }^2 \\
		&= o_p(1) \sup_{\mo \in \Mo} \frac{1}{n \R[\mo]} E\norm{ \bm z}^2 \\
		&= o_p(1) \sup_{\mo \in \Mo} \frac{\d}{n \R[\mo]} \limp 0  
	\end{align*}
	where the first equality is due to (\ref{eq:consistJ}) in Assumption~\ref{ass:secondOrderLLN}, the second equality is due to Assumption~\ref{ass:eig}, and the last equality is guaranteed by assumption~(\ref{eq93}).

	Next, we prove (\ref{eq82}) with $k=6$. 
	Applying Lemma~\ref{lemma:normality}, we could rewrite  
	\begin{align}
		\frac{|A_6[\mo]|}{\R[\mo]} = A_7[\mo] + A_8[\mo] + A_9[\mo] ,	\nonumber
	\end{align}
	where we define 
	\begin{align*}
		A_7[\mo]	&= \frac{\norm{\bm w_n[\mo]}_{V_n(\thT[\mo]; \mo)^{-1}}^2 - \tr \bigl\{ V_n(\thT[\mo]; \mo)^{-1} J_n(\thT[\mo]; \mo) \bigr\}}{n \R[\mo]}, \\
		A_8[\mo] &= \frac{\norm{\bm w_n[\mo]}_{\nu_{1,n}[\mo]}^2 }{n \R[\mo]}, \quad 
		A_9[\mo] = \frac{\bm \nu_{2,n}[\mo]^\T \bm w_n[\mo] }{n \R[\mo]} . 
	\end{align*}
	Using assumption (\ref{eq89}) and similar arguments as in (\ref{eq90}), we can prove $\limsup_{\mo \in \Mo} |A_7[\mo] | \limp 0$.
	Similarly, because  
	\begin{align}
		|A_8[\mo]| &= o_p(1) \frac{\norm{\bm w_n[\mo]}^2 }{n \R[\mo]} \nonumber	
	\end{align}
	where $o_p(1)$ is uniform in $\Mo$, assumption~(\ref{eq88}) guarantees that $\sup_{\mo \in \Mo} A_8[\mo] \limp 0$. 
	 Cauchy inequality and assumption~(\ref{eq88}) also imply that  
	\begin{align}
		\sup_{\mo \in \Mo}|A_9[\mo] | \leq \sup_{\mo \in \Mo}\frac{\norm{\bm \nu_{2,n}[\mo]} \times \norm{\bm w_n[\mo]} }{n \R[\mo]} \limp 0. 
	\end{align}

	Finally, we prove (\ref{eq83}). 
	From (\ref{eq80}) and $\tau$-consistency of $\thE[\mo]$, 
	we have 
	\begin{align*}
	\L[\mo] 
	&= \E \l(\cdot, \thE[\mo]; \mo) \\
	&= \E \l(\cdot, \thT[\mo]; \mo) + n^{-2\tau} O_p(1)
	\end{align*}
	where $O_p(1)$ is uniformly in $\Mo$. 
	Therefore 
	\begin{align*}
		\sup_{\mo \in \Mo} \frac{\L[\mo]}{\R[\mo]}	
		&= 1+ \sup_{\mo \in \Mo} \frac{\L[\mo]-\E \L[\mo]}{\R[\mo]} \\
		&= 1+O_p(1) \sup_{\mo \in \Mo} \frac{1}{n^{2\tau}\R[\mo]} \limp 1. 
	\end{align*}
	

\section{Proof of Corollary~\ref{coro:linear}}

A sketch of the proof is outlined below. 
We only need to verify Assumptions~2 to 7.
Assumption~4 is implied by the assumption that $X$ are independent and $V_n(\thT ; \mo)=2 \S_{xx}=2I$. 
Due to the boundedness condition $\norm{\thT[\mo]}=\norm{(\S_{xx}[\mo])^{-1} \S_{x\mu}[\mo]} < c \sqrt{\p}$ for some constant $c$. We choose $\H[\mo]$ to be $\{\th \in \mathbb{R}^{\d}: \norm{\th-\thT[\mo]} < c \sqrt{\p}\}$.  
We choose any fixed $\tau$ satisfying 
\begin{align}
\max\biggl\{{\co \nu+3w}, \frac{\zeta}{2} \biggr\} < \tau \leq \frac{1}{2} - w. \label{tau}
\end{align}

For Assumption~2, 
$$
\E \ell \bigl(\cdot, \th ; \mo \bigr) - \E \ell \bigl(\cdot, \thT[\mo] ; \mo \bigr)
=\norm{\th-\thT[\mo]}_{\S_{xx}[\mo]}^2 \geq \varepsilon^2
$$
for all $\norm{\th-\thT[\mo]} \geq \varepsilon$.
Moreover, 
$
E_n \ell \bigl(\cdot, \th ; \mo \bigr) - \E \ell \bigl(\cdot, \th ; \mo \bigr)
$
has mean $0$ and variance $n^{-1}\var\{(Y-\th^\T \bm X[\mo])^2\} = O(\p^2/n)=o(1)$ uniformly in $\th \in \H[\mo]$ and $\mo\in \Mo$.

For Assumption~3, 
$E_n \bm \psi_{n}(\cdot, \th ; \mo) - \E \bm \psi_{n}(\cdot, \th ; \mo)$ has mean zero and covariance $n^{-1} \var\{(Y - \th^\T X[\mo]) X[\mo]^\T \}$. 
Let $A = \E\{ (Y - \th^\T X[\mo])^2 X[\mo] X[\mo]^\T \}$.
Let $\norm{\cdot}_F$ denote the Frobenius norm.
Since
\begin{align*}
&\biggl\lVert n^{-1} \var\{(Y - \th^\T X[\mo]) X[\mo]^\T \} \biggr\rVert 
\leq n^{-1}\norm{A} \\
&= n^{-1} O(\p) \norm{\E \{X[\mo] X[\mo]^\T \}} \\
&\leq n^{-1} O(\p) \norm{\E \{X[\mo] X[\mo]^\T\}}_F 
\leq n^{-1} O(\p^2)
\end{align*}
uniformly in $\th \in \H[\mo]$ and $\mo\in \Mo$.
Thus any $\tau $ satisfying (\ref{tau}) suffices.

For Assumption~5, 
\begin{align*}
	&\norm{\bm \psi_{n}(\bm z, \th_1; \mo)-\bm \psi_{n}(\bm z, \th_2; \mo)	}
	= \norm{ X[\mo] X[\mo]^\T (\th_1 - \th_2) } \\
	&\leq \norm{X[\mo] X[\mo]^\T}_F \norm{\th_1 - \th_2}
	\leq c \p \norm{\th_1 - \th_2}.
\end{align*}
So $m_n[\mo]=c \p$ suffices. 
This together with the condition {\co $\nu+3w < \tau$} implies (\ref{eq:mCond}).

For Assumption~6, similar as before, it can be shown that
$\norm{\hat{J}_n(\th ; \mo) - J_n(\th ; \mo) }=O(n^{-1/2}d^2) = o(1)$,
and $\norm{\hat{V}_n(\th ; \mo) - V_n(\th ; \mo) } = o(1)$ uniformly in $\th \in \H[\mo]$ and $\mo\in \Mo$.
Also, $V_n(\thT ; \mo) - V_n(\th ; \mo) =0$.

For Assumption~7,
(\ref{eq84}) is implied by $\zeta/2 < \tau$, and
(\ref{eq93}) is implied by $\zeta < 1 - w$.
Since 
\begin{align*}
&\E \bigl\{\l(\cdot, \thT[\mo] ; \mo)- \E \l(\cdot, \thT[\mo]; \mo) \bigr\} \\
&\leq \E \bigl\{\l(\cdot, \thT[\mo] ; \mo) \bigr\}^2 = O(\p^2),
\end{align*}
(\ref{eq85}) holds under $m_1=1$ and $\zeta < 1-w$.
Similar calculations as before show that (\ref{eq89}) and (\ref{eq88}) are implied by $\zeta < 1-2w$ and $m_2=m_3=1$.


\section{Proof of Theorem~\ref{thm:graphexpert}}
First, we introduce the concept of ``compound experts''. A compound expert is defined as an expert sequence $(i_1,i_2,\ldots,i_T)$ whose $\size \leq k$ with some prescribed $k>0$. Then in order to tackle the problem of ``tracking the best expert'', we could simply apply the exponentially re-weighting algorithm over all the possible compound experts, which can yield provable tight regret bounds. The reason why this simple strategy is not used in practice is that the number of compound experts is usually too large to manage, while the fixed share algorithm greatly reduces the computational complexity and has similar regret bounds.

For our extension of ``tracking the best expert'' with graphical transitional constraints, following a similar proving strategy used in \cite[Chapter 5]{cesa2006prediction}, we first prove an equivalence between the results of the exponentially re-weighting algorithm over compound experts and the algorithm that we propose, and then apply the regret bound for the former algorithm directly.

The exponentially re-weighting algorithm that we are considering here is as follows. At each time $t=0,1,\ldots,T$, the distribution over the compound experts is maintained by $w'_t(i_1,i_2,\ldots,i_T)$ (not necessarily normalized) for all the sequences $(i_1,i_2,\ldots,i_T)$. The initial distribution is
\begin{align*}
&w'_0(i_1,i_2,\ldots,i_T) \\
&= w'_0(i_1)w'_0(i_2|i_1)w'_0(i_3|i_1,i_2)\cdots w'_0(i_T|i_1,\ldots,i_{T-1})\\
&= w'_0(i_1)w'_0(i_2|i_1)w'_0(i_3|i_2)\cdots w'_0(i_T|i_{T-1}) \\
&= w'_0(i_1)\prod_{t=1}^{T-1}w'_0(i_{t+1}|i_t)\\
&= \ind{i_1=1}\prod_{t=1}^{T-1}\biggl[(1-\kappa\beta_{i_t})\ind{i_{t+1}=i_t}+\kappa\beta_{i_t,i_{t+1}}\ind{i_{t+1}\neq i_t}\biggr] ,
\end{align*}
where the second equality is due to Markovian property. 
This initial distribution over compound experts ensures that only the ``valid'' expert sequences (those follow graphical transitions) have positive probabilities.
Based on the exponentially re-weighting updating rule, the distribution at each time instant $t=1,2,\ldots,T$ becomes
$
w'_t(i_1,i_2,\ldots,i_T) =  w'_0(i_1,i_2,\ldots,i_T)\exp(-\eta\sum_{s=1}^t l(i_s,\ys)) 
$. 

Marginally, at time $t$,
$$
w'_{i,t} = \sum_{i_1,\ldots,i_t,i,i_{t+2},\ldots,i_T} w'_t(i_1,\ldots,i_t,i_{t+2},\ldots,i_T).
$$ 
Then we have
$p'_{i,t} = w'_{i,t}/W'_t$ with  $W'_t=\sum_{j=1}^N w'_{j,t}$,
and $p'_{i,0} = w'_{i,0} = \ind{i=1}$. 
The exponentially forecaster draws action according to expert $i$ at time $t+1$ with probability $p'_{i,t}$.

\begin{lemma}
\label{lemma:graph1}
For all $\kappa\in(0,1/\deg)$, for any sequence of $T$ outcomes, and for all $t=0,1,\ldots,T$, the predictive distribution $p_{i,t}$ for $i=1,\ldots,N$ generated by our proposed Algorithm \ref{algo:trackGraphOrigin} is the same as the predictive distribution $p'_{i,t}$ for $i=1,\ldots,N$ that is maintained by the special exponentially re-weighting algorithm described above.
\end{lemma}

\begin{proof}
It is enough to show that for all $i$ and $t$, $w_{i,t}=w'_{i,t}$. We proceed by induction on $t$. For $t=0$, $w_{i,0}=w'_{i,0}=\ind{i=1}$ for all $i$. For the induction step, assume that $w_{i,s}=w'_{i,s}$ for all $i$ and all $s<t$. We then have
\begin{align*}
w'_{i,t}
=&\sum_{i_1,\ldots,i_t,i_{t+2},\ldots,i_T} w'_t(i_1,\ldots,i_t,i,i_{t+2},\ldots,i_T)\\
=&\sum_{i_1,\ldots,i_t,i_{t+2},\ldots,i_T} e^{ -\eta\sum_{s=1}^t l(i_s,\ys) } \times \\
&\hspace{2cm} w'_0(i_1,\ldots,i_t,i,i_{t+2},\ldots,i_T)\\
=&\sum_{i_1,\ldots,i_t} e^{ -\eta\sum_{s=1}^t l(i_s,\ys) } w'_0(i_1,\ldots,i_t,i)\\
=&\sum_{i_1,\ldots,i_t} e^{ -\eta\sum_{s=1}^t l(i_s,\ys) } w'_0(i_1,\ldots,i_t)\frac{w'_0(i_1,\ldots,i_t,i)}{w'_0(i_1,\ldots,i_t)}\\
=&\sum_{i_1,\ldots,i_t} e^{ -\eta\sum_{s=1}^t l(i_s,\ys) } w'_0(i_1,\ldots,i_t) \times \\
&\hspace{1cm} \biggl[(1-\kappa\beta_{i_t})\ind{i=i_t}+\kappa\beta_{i_t,i}\ind{i\neq i_t}\biggr]\\
=&\sum_{i_1,\ldots,i_t} e^{ -\eta l(i_t,\yt) } \exp\biggl(-\eta\sum_{s=1}^{t-1} l(i_s,\ys)\biggr) \times  \\
&\hspace{1cm} w'_0(i_1,\ldots,i_t)\biggl[(1-\kappa\beta_{i_t})\ind{i=i_t}+\kappa\beta_{i_t,i}\ind{i\neq i_t}\biggr]\\
=&\sum_{i_t} e^{ -\eta l(i_t,\yt) } w'_{i_t,t-1} \biggl[(1-\kappa\beta_{i_t})\ind{i=i_t}+\kappa\beta_{i_t,i}\ind{i\neq i_t}\biggr].
\end{align*}
By induction hypothesis, $w'_{i,t}$ further equals 
\begin{align*}
&\sum_{i_t} e^{ -\eta l(i_t,\yt) } w_{i_t,t-1}  
  \biggl[(1-\kappa\beta_{i_t})\ind{i=i_t}+\kappa\beta_{i_t,i}\ind{i\neq i_t}\biggr] \\
&=\sum_{i_t} v_{i_t,t-1}\biggl[(1-\kappa\beta_{i_t})\ind{i=i_t}+\kappa\beta_{i_t,i}\ind{i\neq i_t}\biggr]\\
&=(1-\kappa \bt_{i}) v_{i,t} + \kappa \sum_{j=1}^N\bt_{ji} v_{j,t} = w_{i,t} 
\end{align*}
where the last equality is by $\beta_{ii}=0$. 
\end{proof}

\begin{lemma}
\label{lemma:graph2}
For all $T\geq1$, if $l\in[0,1]$ and we run the exponentially weighted forecaster over compound experts as described before, we will have
\begin{align*}
\sum_{t=1}^T\sum_{i=1}^N p'_{i,t}l(i,\yt)\leq \frac{1}{\eta}\ln\frac{1}{W'_T}+\frac{\eta}{8}T
\end{align*}
\end{lemma}

\begin{proof}
First, notice that
\begin{align*}
W'_t
=&\sum_{i=1}^N w'_{i,t}\\
=&\sum_{i=1}^N \sum_{i_1,\ldots,i_t,i_{t+2},\ldots,i_T}^N w'_t(i_1,\ldots,i_t,i,i_{t+2},\ldots,i_T)\\
=&\sum_{i_1,\ldots,i_T} w'_t(i_1,\ldots,i_T).
\end{align*}
Then, we also have
\begin{align*}
&\sum_{i=1}^N p'_{i,t}l(i,\yt) 
= \sum_{i_{t}} l(i_t,\yt)\frac{w'_{i_t,t}}{W'_{t-1}}\\
&= \sum_{i_{t}} l(i_t,\yt)\frac{\sum_{i_1,\ldots,i_{t-1},i_{t+1},\ldots,i_T}w'_{t-1}(i_1,\ldots,i_T)}{W'_{t-1}}\\
&= \sum_{i_1,\ldots,i_T}\frac{w'_t(i_1,\ldots,i_T)}{W'_{t-1}} l(i_t,\yt) .
\end{align*}
Then we can directly apply Lemma 5.1 in \cite[Chapter 5]{cesa2006prediction} by noticing that $W'_0=1$.
\end{proof}

\vspace{0.2cm}

\noindent\textbf{Proof of Theorem~\ref{thm:graphexpert}}

\begin{proof}
According to Lemma \ref{lemma:graph1}, it is equivalent to prove the bound for the equivalent exponentially weighted forecaster. There we have
\begin{align*}
&w'_0(i_1,\ldots,i_T) \\
&= \ind{i_1=1}\prod_{t=1}^{T-1}\biggl[(1-\kappa\beta_{i_t})\ind{i_{t+1}=i_t}+\kappa\beta_{i_t,i_{t+1}}\ind{i_{t+1}\neq i_t}\biggr]\\
&\geq  (1-\kappa \deg)^{T-k-1} \kappa^k
\end{align*}
for all the sequence $(i_1,\ldots,i_T)$ with $\size\leq k$ and transitions restricted on the graph.

Also, we have
\begin{align*}
\ln w'_T(i_1,\ldots,i_T)=\ln w'_0(i_1,\ldots,i_T)-\eta\sum_{t=1}^T l(i_t,\yt) .
\end{align*}
And $W'_T\geq w'_T(i_1,\ldots,i_T)$. Then by Lemma \ref{lemma:graph2} and some simple manipulations, we obtain
\begin{align*}
&\sum_{t=1}^T \biggl(\sum_{i=1}^N l(i,\yt)p_{i,t}-l(i_t,\yt)\biggr) \\
&\leq \frac{1}{\eta}(T-k-1)\log\frac{1}{1-\kappa \deg}+\frac{1}{\eta}k\log\frac{1}{\kappa}+ \eta\frac{T}{8}.
\end{align*}
\end{proof}

\ifCLASSOPTIONcaptionsoff
  \newpage
\fi

\clearpage

\balance
\bibliographystyle{IEEEtran}
\bibliography{learningLimit,references_com}

\end{document}